\RequirePackage{fix-cm}
\documentclass[smallextended]{svjour3}  
\smartqed  
\usepackage{amsmath,amsfonts,amssymb,graphicx, hyperref, xcolor, float}
\usepackage{mathtools}
\usepackage{algorithm} 
\usepackage{algorithmic}  
\usepackage[algo2e,vlined,linesnumberedhidden,ruled,resetcount]{algorithm2e}
\usepackage{booktabs}
\usepackage[normalem]{ulem}
\usepackage{multirow}
\usepackage{soul}
\definecolor{amber}{rgb}{1.0, 0.49, 0.0}
\DeclareMathAlphabet{\mathpzc}{OT1}{pzc}{m}{it}
\usepackage{enumitem}
\newtheorem{assumption}{Assumption}
 
\usepackage{setspace}

\newcommand*{\xbf}{{\mathbf x}}
\newcommand*{\ubf}{{\mathbf u}}
\newcommand*{\ybf}{{\mathbf y}}

\newcommand*{\zbf}{{\mathbf z}}
\newcommand*{\gbf}{{\mathbf g}}
\newcommand*{\vbf}{{\mathbf v}}
\newcommand*{\wbf}{{\mathbf w}}

\newcommand*{\R}{{\mathbb R}}

\newcommand*{\0}{{\mathbf 0}}
\newcommand*{\1}{{\mathbf 1}}

\newcommand*{\E}{{\mathbb E}}

\usepackage{ulem}
\begin{document}

\title{A Dimension-Insensitive Algorithm for Stochastic Zeroth-Order Optimization
\thanks{This work is partially supported by  UF AI Catalyst  Grants and NSF grants CMMI-2016571. 
}
}

\author{Hongcheng Liu \and Yu Yang}
\institute{ 
           H. Liu,   Y. Yang (Corresponding author)\at
            Department of Industrial and Systems Engineering, University of Florida \\
            \email{liu.h@ufl.edu,  yu.yang@ise.ufl.edu}
}

\date{Received: date / Accepted: date}

\maketitle
\begin{abstract}
This paper concerns a convex, stochastic zeroth-order optimization (S-ZOO) problem.  The objective is to minimize the expectation of a cost function whose gradient is not directly accessible. For this problem, traditional optimization algorithms mostly yield query complexities that grow polynomially with dimensionality (the number of decision variables). Consequently, these methods may not perform well in solving massive-dimensional problems arising in many modern applications. Although more recent methods can be provably dimension-insensitive, almost all of them require arguably more stringent conditions such as everywhere sparse or compressible gradient.  In this paper,  we propose a sparsity-inducing stochastic gradient-free (SI-SGF) algorithm, which provably yields a dimension-free (up to a logarithmic term) query complexity in both convex and strongly convex cases. Such  insensitivity to the  dimensionality  growth  is proven, for the first time, to be achievable when neither gradient sparsity nor gradient compressibility is satisfied. Our numerical results demonstrate a consistency between our theoretical prediction and the empirical performance.


\end{abstract}
\sloppy
\section{Introduction}\label{section: intro}

For many modern   optimization problems, the (stochastic) gradient can be  hardly available. This happens, for instance, when the objective function   admits no known explicit form, or the (stochastic) gradient is too expensive to compute.
 Applications of this type render many efficient  and thus popular algorithms, such as the stochastic first-order methods, no longer directly applicable. As a remedy, zeroth-order optimization (ZOO), also known as black-box or derivative-free optimization \cite{cai2020zeroth},  has attracted much research interest.

In this paper, we propose a novel zeroth-order method to solve 
a stochastic ZOO (S-ZOO)  problem with the following formulation:
\begin{align}
\min_{\xbf\in\R^d}\,\{F(\xbf):=\E\left[f(\xbf,\xi)\right]\},\label{SP problem}
\end{align}
where  $\xi$ is a random vector of problem parameters whose probability distribution $\mathbb P$ is supported on a measurable set $\Theta\subseteq\R^q$, and $f:\,\R^d\times\Theta\rightarrow\R$ is deterministic and measurable. Denote by $\xbf^*\in\R^d$ an optimal solution to \eqref{SP problem}.  Here, the dimensionality of the problem $d$ is assumed, without loss of generality, to satisfy $d\geq 3$ throughout this paper. In addition,  it is assumed that $f(\,\cdot\,,\xi)$ is everywhere continuously differentiable for almost every $\xi\in\Theta$, $F$ is convex, and  the expectation   $\E\left[f(\xbf,\xi)\right]=\int_{\Theta} f(\xbf,\xi)\,\text{d}\mathbb P(\xi)$ is well defined and finite-valued for every $\xbf\in\R^d$.   Given $\xi$,  let $\nabla f(\,\cdot\,,\xi)$ be the gradient of $f(\,\cdot\,,\xi)$.    $\Vert\cdot\Vert_1$ and $\Vert\cdot\Vert$ are the 1-norm and 2-norm, respectively. Furthermore, 
 we  impose  the following assumptions hereafter for some known constant $R\geq 1$. 
  \begin{assumption}\label{sample simulation}
 It is possible to generate independent and identically distributed (iid) realizations, $\xi_1,\,\xi_2,....,$ of the random vector $\xi$.
 \end{assumption}
 
  \begin{assumption}\label{assumption oracle}
 There is a stochastic zeroth-order oracle that returns the value of $f(\xbf,\xi)$ for a given input point $(\xbf,\xi)\in\R^d\times\Theta$.
\end{assumption}

\begin{assumption}\label{mean and variance}
For  every $\xbf\in\R^d$, 
 it holds that $\nabla F(\xbf)=\mathbb E[\nabla f(\xbf,\xi)]$ and
$
 \E\left[\left\Vert \nabla f(\xbf,\xi)-\nabla F(\xbf)\right\Vert^2\right]\leq \sigma^2
$ for some $\sigma>0$. 
\end{assumption}
\begin{assumption}\label{Lipschitz}
There exists a constant $L>0$, such that \[\Vert\nabla f(\xbf_1,\xi)-\nabla f(\xbf_2,\xi)\Vert\leq L \Vert\xbf_1-\xbf_2\Vert,\]  for all 
$\xbf_1,\,\xbf_2\in\R^d$ and almost every $\xi\in\Theta$.
\end{assumption}

\begin{assumption}\label{bounded}
Problem \eqref{SP problem} admits a bounded optimal solution such that $\{\xbf:\,\Vert \xbf\Vert_1\leq R\} \cap \arg\min_{\xbf\in\R^d}\,F(\xbf) \neq \emptyset$.
\end{assumption}

Assumptions \ref{sample simulation} through \ref{Lipschitz} above are common in the ZOO literature (See \cite{ghadimi2013stochastic,balasubramanian2018zeroth,nesterov2017random}).  Assumption \ref{sample simulation} allows 
for the availability of a simulator to generate sample scenarios of the random vector $\xi$.   Assumption \ref{assumption oracle} concerns the algorithmic oracle.
By this assumption, we may only have access to noisy
objective values $f(\xbf,\xi)$, i.e.,  inexact zeroth-order information of $F$, for a given tuple of function input $(\xbf,\xi)$. No higher-order information, such as gradient or hessian, is available.   Assumption \ref{mean and variance} stipulates that $\nabla f(\,\cdot\,,\xi)$ is an unbiased estimator of $\nabla F$ with a bounded variance. Assumption \ref{Lipschitz} requires $f(\,\cdot\,,\xi)$ to be differentiable and its gradient to be Lipschitz continuous. A well-known inequality as an immediate result of   this assumption  is that, for almost every $\xi\in\Theta$, and for all $\xbf,\,\ybf\in\R^d$:
\begin{align}
f(\xbf,\,\xi)\leq f(\ybf,\,\xi)+\langle\nabla f(\ybf,\xi),\,\xbf-\ybf\rangle+\frac{L}{2}\Vert\xbf-\ybf\Vert^2.\label{smooth well-known general}
\end{align}

Assumption \ref{bounded} imposes the  boundedness of an optimal solution. While this assumption  also  holds for many problems in practice, we are particularly interested in  scenarios where the problem dimensionality $d$ is very large compared to   $R$; that is, $R\ll d$. In this case, Assumption \ref{bounded} is also referred to as the weak sparsity condition by \cite{negahban2012unified,li2018minimax} in  statistics and inverse problems, which is an extension to the conventional sparsity.  Indeed, when the optimal solution $\xbf^*$ is a sparse element of a hypercube (that is, $\xbf^*\in[-r,\,r]^d:\,\Vert\xbf^*\Vert_0=s\ll d$ for some  non-negative integer $s$  and some  scalar\footnote{Here $\Vert\cdot\Vert_0$ denotes the number of nonzero entries of ``$\,\cdot\,$''.} $r>0$), we may see that weak sparsity easily holds when the traditional sparsity holds, as $R=s\cdot r\ll d$.   

However, our results to be presented subsequently may not be advantageous  when $R$ is large or even comparable with $d$. Admittedly, one may argue that, regardless of how large $R$ is,  we can always introduce a change of variables $\zbf:=\xbf/C_R$, for some quantity $C_R$ dependent only on $R$, such that $\zbf$ becomes the actual vector of decision variables and  $\Vert \zbf\Vert_1$ is small (thus, weak sparsity still holds for $\zbf$). However, readers are reminded that such rescaling may undesirably affect both the  variance  $\sigma^2$ in Assumption \ref{mean and variance} and the Lipschitz constant $L$ in Assumption \ref{Lipschitz} --- both $\sigma^2$ and $L$ will grow polynomially in $C_R$ after the rescaling.

Some of our results will be  additionally contingent upon the assumptions of strong convexity and (the traditional) sparsity  as below:

 \begin{assumption}\label{strongly convex assumption}
Function $F(\cdot)$ is strongly convex with modulus $\mu>0$.
 \end{assumption}

 

 \begin{assumption}\label{sparsity assumption}

Problem \eqref{SP problem} admits a finite, $s$-sparse optimal solution. More specifically,  there exists $\xbf^*\in \{\xbf:\,\Vert \xbf\Vert_1\leq R\}\cap\arg\min_{\xbf\in\R^d}\,F(\xbf)$  such that $\Vert \xbf^*\Vert_0\leq s$, for some   $s:\,1\leq s\ll d$.
 \end{assumption}

Strongly convex functions under Assumption \ref{strongly convex assumption} have been frequently  studied  in function minimization. Assumption \ref{sparsity assumption} is the conventional sparsity condition, which is a more stringent requirement than Assumption \ref{bounded}. This condition holds for many modern statistical and machine learning problems as discussed, e.g., by \cite{negahban2012unified,candes2007dantzig,fan2001,bickel2009simultaneous}. Sparsity and its benefit in decision-making  and optimization problems have been discussed by much, and growingly more, literature, e.g., in \cite{jordan1995principles,letchford2009exploiting,cho2001reduction}. Exploiting sparsity in stochastic optimization has also been studied by \cite{liu2019sample}. Problem \eqref{SP problem}, even  under  both Assumptions \ref{strongly convex assumption} and \ref{sparsity assumption} additionally,  has   a wide spectrum of  applications, such as simulation-based optimization \cite{rubinstein2016simulation}, parameter tweaking of deep learning models \cite{snoek2012practical}, and optimal therapeutic designs \cite{marsden2008computational}.

Effective algorithmic paradigms for solving \eqref{SP problem} are available in the rich ZOO literature, including pattern search \cite{torczon1997convergence,hooke1961direct,nelder1965simplex}, random search \cite{solis1981minimization}, and bayesian optimization \cite{mockus2012bayesian}, among many others (see \cite{larson2019derivative} for an excellent review).  Among the existing ZOO methods, the gradient estimation-based ZOO  framework discussed by  seminal works such as \cite{nesterov2017random,spall1998overview,agarwal2010optimal,duchi2015optimal,ghadimi2013stochastic} is closely related to this current work. 

Despite numerous results on ZOO, a persistent challenge, as pointed out by \cite{cai2020zeroth,balasubramanian2018zeroth},  is that the performance of almost all existing ZOO algorithms deteriorates rapidly as the problem dimensionality $d$ increases. In particular, for convex S-ZOO with  a potentially nonsmooth cost function (a more general setting than ours in terms of Assumption \ref{Lipschitz} above), a randomized gradient-free algorithm achieves a  complexity  of $O(d^2/\epsilon^2)$-many queries of the zeroth-order oracles\footnote{The query complexity of the zeroth-order oracle refers to the number of calls to the zeroth-order oracle required to achieve a desired accuracy  $\epsilon>0$.}  according to \cite{nesterov2017random}. If the smoothness condition as in Assumption \ref{Lipschitz} holds,  \cite{ghadimi2013stochastic}  provides a rate of $O(D_0d/\epsilon^2)$, where $D_0$ is the squared Euclidean distance between the initial solution and the optimal solution. Some  analysis on the performance lower bound \cite{jamieson2012query} indicates that the rate by \cite{ghadimi2013stochastic} is already optimal without additional regularity assumptions on the objective function $F$. These  complexity results suggest the potential inefficiency of existing ZOO algorithms for high-dimensional applications, where the number of decision variables can be in millions, billions, or even more. On the other hand, such high-dimensional problems are emerging rapidly in, e.g., data science, deep learning, and imaging, due to the ever-increasing demand for higher resolution and improved comprehensiveness in an optimized system. 

Although several  promising high-dimensional ZOO paradigms have been proposed recently, e.g., by \cite{wang2018stochastic,balasubramanian2018zeroth,cai2020zeroth,balasubramanian2018zeroth_neurips},  the corresponding ZOO  theories are based on
some arguably restrictive assumptions. Indeed, while query complexities that are (notably) logarithmic in $d$ have been achieved by \cite{wang2018stochastic,balasubramanian2018zeroth,balasubramanian2018zeroth_neurips}, their results are based on the assumption that $\nabla F$, the gradient of $F$, is everywhere $s$-sparse for some $s\ll d$. This means that there are always no more than $s$-many nonzero components in the gradient  vector $\nabla F(\xbf)$, for any choice of  $\xbf$. Some results by \cite{balasubramanian2018zeroth}  further require  that the optimal solution $\xbf^*$ is sparse.  The assumption of  sparse gradient, according to \cite{cai2020zeroth}, is comparatively stringent. In relaxing this assumption, \cite{cai2020zeroth}  has developed the zeroth-order regularized optimization (ZORO) method, which is effective when the gradient is dense and satisfies a compressibility condition   proposed therein. Additionally, \cite{cai2020zeroth}  imposes a   more specific problem structure than \eqref{SP problem}---the random noise in evaluating the zeroth-order information is additive. Namely, it is assumed that $f(\xbf)=F(\xbf)+u$ for some random variable $u\in\R$ with a bounded support\footnote{A more general problem with a known, nonsmooth regularization term has also been considered  by  \cite{cai2020zeroth}.}. 

In contrast to the aforementioned methods, this paper presents a novel, sparsity-inducing stochastic gradient-free (SI-SGF) algorithm, which can effectively reduce the query complexity in terms of the dependence on the dimension $d$, even when  most of the aforementioned assumptions, i.e.,  sparse gradient, compressible gradient, or additive randomness, are absent. Imposed instead in this work is the  more common and more easily verifiable assumption on the (weak) sparsity level  of the optimal solution $\xbf^*$. More specifically, our main result  in Theorem \ref{main theorem} only requires  a weak sparsity assumption as in Assumption \ref{bounded}, which holds even if $\xbf^*$ is dense. When  $R$ therein is dimension-independent, we prove that the SI-SGF can yield a dimension-free (up to a logarithmic term) query complexity.  A significant acceleration  is further achieved when $\xbf^*$ is sparse (as in Assumption \ref{sparsity assumption}) and  the objective function of \eqref{SP problem} is strongly convex (as in Assumption \ref{strongly convex assumption}).
%
Table \ref{table one} summarizes the complexity results and assumptions for  the proposed SI-SGF and several important alternatives.
 Although the complexity rates by \cite{cai2020zeroth,balasubramanian2018zeroth,balasubramanian2018zeroth_neurips} can be more appealing than ours in terms of the desired accuracy $\epsilon$, the proposed SI-SGF is perhaps the first algorithm that can be shown to achieve dimension-insensitive query complexities, when gradient is neither   sparse nor   compressible. 

\begin{table}
\caption{Comparison of query complexity results. The ``Assumption'' column presents conditions other than Assumptions \ref{sample simulation} through \ref{bounded}, which are standard to the convex ZOO literature. $D_0:=\Vert \xbf^1-\xbf^*\Vert^2$ measures the squared distance between the initial solution and an optimal solution.  Although $D_0\sim \mathcal O(d)$ in general, 
it  can be $\mathcal O(s)$ when $\xbf^*$ has only $s$-many nonzero components and the initial solution is chosen to be sparse (e.g., the initial solution can be the all-zero vector). ``$s$-sparse gradient'' refers to the assumption that the gradient has no more than $s$-many nonzero components everywhere, and ``$\xbf^*$ is $s$-sparse'' means that the optimal solution has no more than $s$-many nonzero components.}\label{table one}
\begin{center}
\small
\begin{tabular}{ ccc c } 
 \toprule

Algorithms    & Complexity & Assumption \\ \midrule\midrule
  \cite{nesterov2017random} & $\mathcal O\left(\frac{d^2}{\epsilon^2}\right)$ &
   $\begin{matrix}
   \text{No additional assumption}
   \\
   \text{Cost function can be nonsmooth}
   \end{matrix}
   $
    \\ \midrule
   \cite{ghadimi2013stochastic}  & $\mathcal O\left(\frac{d D_0}{\epsilon^2}\right)=\mathcal O\left(\frac{ d^2}{\epsilon^2}\right)$ & No additional assumption  \\  \midrule
      \cite{ghadimi2013stochastic}  & $\mathcal O\left(\frac{d D_0}{\epsilon^2}\right)=\mathcal O\left(\frac{ ds}{\epsilon^2}\right)$ & $\xbf^*$ is $s$-sparse  \\  \midrule
   
   \cite{wang2018stochastic} 
   & $\mathcal O\left(\frac{s(\ln d)^3}{\epsilon^3}\right)$ & $\begin{matrix}
   \text{$s$-sparse gradient}
   \\
   \text{Bounded 1-norm of gradient}
   \\
   \text{Bounded 1-norm of Hessian}
   \\
   \text{Additive randomness}
   \\
   \text{Function sparsity}
   \\
   \Vert \xbf^*\Vert_1\leq R
   \end{matrix}
   $
   \\ 
  \midrule
   \cite{cai2020zeroth} & $\mathcal O\left(s\cdot \ln d\cdot \ln\left(\frac{1}{\epsilon}\right)\right)$ &
   $\begin{matrix}
   \text{Compressible gradient}
   \\
   \text{Bounded 1-norm of Hessian}
   \\
   \text{Restricted strong convexity}
   \\
   \text{Additive randomness}
   \\
   \text{Coercivity}
   \end{matrix}
   $
    \\ \midrule
    \cite{balasubramanian2018zeroth,balasubramanian2018zeroth_neurips} & $\begin{matrix}\mathcal O\left(\left(\frac{D_{0}s^2}{\epsilon}+\frac{D_0 s}{\epsilon^2}\right)(\ln d)^2\right)
    \\
    \mathcal =\,\mathcal O\left(\left(\frac{s^3}{\epsilon}+\frac{s^2}{\epsilon^2}\right)(\ln d)^2\right)
    \end{matrix}$ 
    & $\begin{matrix}
   \text{$s$-sparse gradient}
   \\
   \text{$\xbf^*$ is $s$-sparse}
   \end{matrix}
   $
    \\ \midrule
   Proposed  & $\mathcal O\left(\frac{(D_0+R)^3\ln d}{\epsilon^{3}}\right)$ & $\Vert\xbf^*\Vert_1\leq R$  \\   \midrule
   Proposed & $\begin{matrix}
   \mathcal O\left(\frac{(s+D_0+R)^2\ln d}{\epsilon^{2} }\right)
   \\=\mathcal O\left(\frac{(s+R)^2
   \ln d}{\epsilon^{2} }\right)\end{matrix}$ & $\begin{matrix}
   \Vert \xbf^*\Vert_1\leq R 
   \\
   \text{$\xbf^*$ is 
   $s$-sparse}
   \\
   \text{Strong convexity}
   \end{matrix}
   $  \\ 
 \bottomrule
 
\end{tabular}
\end{center}
\end{table}

Note that our results do not contradict with the lower performance bounds by \cite{duchi2015optimal} (in Propositions 1 and 2 therein) for a convex ZOO.  While these lower bounds are tight   when the domain is an $\ell_2$-ball, the problem of interest under Assumption \ref{bounded} concerns a special case of their results; that is, when the domain is an $\ell_1$-ball. In our case, the lower performance bounds by \cite{duchi2015optimal} actually becomes ``0''.  Furthermore, our research is focused on making use of some special and important  problem structures in accelerating S-ZOO.  Exploiting special  problem structures to outperform the worst-case theoretical lower bounds is  fairly common   in the optimization literature (e.g., in \cite{nesterov2005smooth}). Although we  hypothesize that   our complexity results are optimal under our setting, we leave the investigation of this hypothesis  for future research.

\subsection{Outline}

The rest of the paper is organized as follows. In section \ref{gradient approximation}, we provide some preliminaries on gradient approximation via randomized smoothing. Section \ref{sec-algorithm} presents the proposed algorithm.  Section \ref{main results sec}  presents  our main complexity results  on the SI-SGF in both convex and strongly convex cases.  A preliminary numerical study is included in Section \ref{sec: numerical}. Finally, Section \ref{sec: conclusions} concludes the paper. Some proofs and auxiliary results  are provided in Appendix \ref{sec: proof}.

\subsection{Notations}
Let $\R$ and $\R_+$ be the collection of all real numbers and non-negative real numbers, respectively.  For any vector $\xbf:=(x_1, \cdots, x_d)^\top \in \R^d$, we sometimes  use $(x_i)$ to denote $(x_1, \cdots, x_d)^\top$ for convenience, and $\xbf^\top $ to denote its transpose. The cardinality of a set $S$ is denoted by $|S|$ and $\xbf_{S}=(x_i:\,i\in S)$ is the subvector of $\xbf$ that only consists of components  in the index set $S$.  $\mathbf 1$ and $\mathbf 0$ are all-one and all-zero vectors of proper dimensions, respectively. 
 $\nabla  F(\xbf)$ is the gradient of $F$ at $\xbf$ and $\nabla_S F(\xbf)$ is the subvector of $\nabla  F(\xbf)$ that only consists of entries from the index set $S$. The set of integers $\{1, 2, \cdots, K\}$ is denoted by $[K]$.  $\lceil \cdot \rceil$ represents  the smallest integer no smaller than ``$\,\cdot\,$''. $\mathcal N_d(\xbf,\Sigma)$ is the $d$-variate normal distribution with mean $\xbf\in\R^d$ and covariance matrix  $\Sigma\in\R^{d\times d}$. Lastly, $\mathcal N(0,1)$ is the standard normal distribution.

  \section{Gradient approximation via randomized smoothing}\label{gradient approximation}
In this section, we provide some  preliminaries  on how to   approximate the gradient of the objective function using only zeroth-order oracles through  a randomized smoothing scheme.  Many  results below   are based on the existing analyses by \cite{nesterov2017random,cai2020zeroth,ghadimi2013stochastic}.  

To approximate the gradient of $f(\,\cdot\,,\xi)$  with respect to $\xbf$, denoted by $\nabla f(\xbf,\xi)$, we propose to follow a similar approach as discussed by \cite{nesterov2017random,cai2020zeroth}, using the  finite-difference-like formula below.
 \begin{align}
   G^{\delta}(\xbf,\xi):= \frac{f(\xbf+\delta \ubf,\xi)-  f(\xbf,\xi)}{\delta}\ubf,\label{randomized smoothing}
 \end{align}
 where $\ubf = (u_i:\,i=1,...,d)$ has iid  entries with $u_i\in\{-1,1\}$, for all $i$, following a discrete uniform  distribution.  
 Hereafter, we denote by  $\E_{\ubf}$ the expectation over $\ubf$ and, in contrast, by $\E$ the expectation over $\xi$. 
By the definition of $\ubf$, we have
 \begin{equation}\label{delta function}
   f^{\delta}(\xbf,\xi):=\E_{\ubf}  \left[f(\xbf+\delta \ubf,\xi) \right]= \frac{1}{2^d}\sum_{\ubf\in\{-1,\,1\}^d} f(\xbf+\delta \ubf,\,\xi).
 \end{equation}
By the probability mass function of $\ubf$, we have that
 \begin{equation}\label{delta function gradient}
    \nabla f^{\delta}(\xbf,\xi) = \frac{1}{2^d}\sum_{\ubf\in\{-1,\,1\}^d}\nabla f(\xbf+\delta\ubf,\xi)=\E_{\ubf}\left[\nabla f(\xbf+\delta\ubf,\xi) \right].    
 \end{equation}
Since $\E[\nabla f(\,\cdot\,,\xi)]=\nabla \E[f(\,\cdot\,,\xi)] $ in our settings,   we have
 \begin{align}
\E[\nabla f^{\delta}(\,\cdot\,,\,\xi)] = \nabla \E[f^{\delta}(\,\cdot\,,\xi)].
 \end{align}

The lemma below provides a characterization on how the randomized smoothing scheme can be effective in approximating both the zeroth- and first-order information of   $f(\,\cdot\,,\xi)$.
 
 \begin{lemma}\label{lemma approx delta} Under Assumption \ref{Lipschitz}, the below statements hold for any   $\delta>0$\emph{:}
 \begin{itemize}
     \item[\emph{(a).}] Let $f^\delta$ be defined as in \eqref{delta function}. Then, for any $\xbf\in\R^d$ and almost every $\xi\in\Theta$,
 \begin{align}
  \vert f^\delta(\xbf,\xi)-f(\xbf,\xi)\vert \leq\frac{L}{2}  d\delta^2.\nonumber
\end{align}
\item[\emph{(b).}] For any $\vbf,\,\xbf\in\R^d$, and almost every $\xi\in\Theta$, 
  \begin{align}
  \left\vert\E_{\ubf}\left[\frac{f(\xbf+\delta\ubf,\xi)-f(\xbf,\xi)}{\delta}\cdot \ubf^\top \vbf\right]- \langle \nabla f(\xbf,\xi),\,  \vbf\rangle\right\vert
    \leq \frac{L \delta d^{3/2}}{2}\Vert \vbf\Vert. \nonumber
 \end{align}

 \end{itemize} 
 \end{lemma}
 \begin{proof}
 The proof of Part (a) is similar to that in \cite{nesterov2017random}, except that $\ubf$ therein follows a different distribution.
In view of $\E_{\ubf}[\ubf]=\0$, we  obtain   that, for almost every $\xi\in\Theta$,
  \begin{align}
   \left\vert f^\delta(\xbf,\xi)-f(\xbf,\xi) \right\vert =\,&   \left|\frac{1}{2^d}\sum_{\ubf\in\{-1,\,1\}^d}\left\{ f(\xbf+\delta\ubf,\xi) -f(\xbf,\xi)-\delta\left\langle \nabla f(\xbf,\xi),\, \ubf\right\rangle\right\}\right|\nonumber
 \\  \leq \,&   \frac{1}{2^d}\sum_{\ubf\in\{-1,\,1\}^d}\left|\left\{ f(\xbf+\delta\ubf,\xi) -f(\xbf,\xi)-\delta\left\langle \nabla f(\xbf,\xi),\, \ubf\right\rangle\right\}\right|\nonumber
  \\  \leq\,&    \frac{1}{2^d}\sum_{\ubf\in\{-1,\,1\}^d}\frac{L}{2}\delta^2 \Vert\ubf\Vert^2,\label{to invoke lipscthiz}
 \end{align}
 where the inequality in \eqref{to invoke lipscthiz} follows from the Lipschitz continuity of $\nabla f(\,\cdot\,,\xi)$. The results in Part (a) immediately follows from the above in view of  $\Vert \ubf\Vert^2= d$, as per the underlying distribution of $\ubf$.
 
For Part (b), similarly, since $\vert f(\xbf+\delta\ubf,\xi)-f(\xbf,\xi)-\langle \nabla f(\xbf,\xi),\,\delta\ubf\rangle\vert\leq
 \frac{L}{2}\Vert \delta \ubf\Vert^2$, we have
 \[  
 \begin{aligned}
  &    \left\vert \E_{\ubf}\left[\frac{f(\xbf+\delta\ubf,\xi)-f(\xbf,\xi)}{\delta}\cdot \ubf^\top \vbf\right]-\E_{\ubf}\left[\langle \nabla f(\xbf,\xi),\,  \ubf\rangle\cdot \ubf^\top \vbf\right]\right\vert\\
    \leq&  \,\E_{\ubf}\left\vert \left[\frac{f(\xbf+\delta\ubf,\xi)-f(\xbf,\xi)}{\delta}\cdot \ubf^\top \vbf\right]-\left[\langle \nabla f(\xbf,\xi),\,  \ubf\rangle\cdot \ubf^\top \vbf\right]\right\vert
\\ \leq&  \frac{L \delta}{2}\Vert  \ubf\Vert^2\vert \ubf^\top \vbf\vert\leq \frac{L \delta}{2}\Vert  \ubf\Vert^3\Vert   \vbf\Vert.
 \end{aligned}
 \]
In view of 
 $\Vert \ubf\Vert^2=d$, and $\E_{\ubf}\left[\langle \nabla f(\xbf,\xi),\,  \ubf\rangle\cdot \ubf^\top \vbf\right] = \langle \nabla f(\xbf,\xi),\vbf\rangle$, we then immediately have the desired result. \qed
 \end{proof}

With \eqref{delta function} and \eqref{delta function gradient}, it is easy to verify the following properties.
 \begin{itemize}
 \item[(a)] For almost every $\xi\in\Theta$, because $\nabla f(\,\cdot\,,\xi)$ is $L$-Lipschitz continuous, so is $\nabla f^\delta(\,\cdot\,,\xi)$.
 \item[(b)] Because  $\E[f(\xbf,\xi)]$ is convex and continuously differnetiable in $\xbf$, so is $F^\delta(\xbf):=\E[f^\delta(\xbf,\xi)]$.
 
 
  \item[(c)] By  the convexity of $F(\,\cdot\,)$, we have  
 \begin{align}F^\delta(\xbf)&=\E\left\{\E_{\ubf}[f(\xbf+\delta\ubf,\xi)]\right\}
 = \E_{\ubf}\left\{\E[f(\xbf+\delta\ubf,\xi)]\right\}\nonumber\\
 &= \E_{\ubf} [F(x+\delta\ubf)]\geq \,F(\xbf)+\E_{\ubf}[\langle \delta \ubf,\,\nabla F(\xbf)\rangle]= F(\xbf).\label{convexity inequality here}
 \end{align}
 
 \item[(d)] Consider the case where $F(\xbf)=\E[f(\xbf,\xi)]$ is strongly convex in $\xbf$ with modulus $\mu$. $F^\delta(\xbf)$ must also be strongly convex. Further invoking Assumption \ref{Lipschitz}, we have, for all $\xbf_1,\,\xbf_2\in\R^d$,
\begin{align}
F^\delta(\xbf_1)-F^\delta(\xbf_2)   \geq&\,  \langle \nabla F^\delta(\xbf_2),\,\xbf_1-\xbf_2 \rangle+\frac{\mu}{2}\Vert\xbf_1-\xbf_2\Vert^2.
 \label{test new results here useful}
\end{align}
 \end{itemize}

\section{The Proposed Sparsity-Inducing Stochastic Gradient-Free (SI-SGF) Algorithm}\label{sec-algorithm}

Our proposed method is shown in Algorithm \ref{main-alg}. At each iteration, it calls the subroutine   in Algorithm \ref{sub-alg}.  In particular, at the $k$-th iteration of Algorithm \ref{main-alg}, $M>0$ is the mini-batch size, $\gamma_k>0$ is the step size,
and $U_k>0$ is a parameter input to Algorithm \ref{sub-alg}. Given the parameter $U\leftarrow U_k$, Algorithm \ref{sub-alg}  takes the input  $\xbf\leftarrow \mathbf x^{k}-\gamma_k \mathbf g_k^{\delta}(\xbf^k)$ and outputs
$\vbf$, which is assigned to $\xbf^{k+1}$ in Algorithm \ref{main-alg}, i.e.,  $\xbf^{k+1}\leftarrow \vbf$.



\begin{algorithm}[H]\small
\SetKwInput{Initialization}{Initialization}
\SetKwInput{Output}{Output}
\Initialization{ Set hyper-parameters $\{\gamma_k\}$, $M$,  $\{U_k\}$, and $K$. Let $\mathbf x^1$ be a fixed feasible solution such that  $\Vert \xbf^1\Vert_1\leq R$ and $\Vert \xbf^1\Vert_0\leq \frac{2R}{U_1}$ (e.g., $\xbf^1:=\0$).}
\For{$k=1,...,K$,}{
\begin{description}[rightmargin=15mm]
\item[Step 1.] Generate a sample mini-batch of size $M$, $(\xi^{k,1},...,\xi^{k,M})$, and compute 
\begin{align}
\mathbf g_k^{\delta}(\xbf^k):=  \,\frac{1}{M}\sum_{m=1}^M \left[\frac{f(\xbf^k+\delta \ubf^{k,m},\,\xi^{k,m})-f(\xbf^k,\xi^{k,m})}{\delta} \ubf^{k,m}\right],\nonumber
\end{align}
where   $\{\ubf^{k,m}\}$ are iid random realizations of the $d$-variate random vector each entry of which follows a discrete uniform distribution on $\{-1,\,1\}$.

\item[Step 2.] Invoke the subroutine in Algorithm \ref{sub-alg}, with input $\mathbf x^{k}-\gamma_k \mathbf g_k^{\delta}(\xbf^k)$,  parameter $U_k$, and output $\xbf^{k+1}$.
\end{description} 
}

\Output{$\xbf^{Y}$ for a random $Y$, which has a  discrete distribution on $[K]$ with a probability mass function $\mathbb P\left[Y=k\right]=\frac{\gamma_{k-1}^{-1}}{\sum_{k=1}^K\gamma_{k-1}^{-1}}$. 
}
\caption{Sparsity-inducing stochastic gradient-free (SI-SGF) algorithm.} \label{main-alg}
\end{algorithm}

Note that the output of the algorithm above  is a randomly drawn element from the algorithm's solution sequence. This output scheme follows  \cite{ghadimi2013stochastic}. We describe   alternative output schemes, which tend to exhibit stronger empirical performance,  in Section \ref{sec alternative output}. 

The design of Algorithm \ref{main-alg} mimics a standard stochastic first-order method (S-FOM), such as in \cite{ghadimi2013stochastic}, except for two differences. First,  we follow \cite{nesterov2017random,cai2020zeroth,ghadimi2013stochastic} in approximating  the stochastic gradient of the S-FOM by a randomized  estimator     as discussed in Section  \ref{gradient approximation} above. Second, we invoke a subroutine to perform sparse projection at each iteration.  The pseudo-code of the  this subroutine is  presented in Algorithm \ref{sub-alg} below.

\begin{algorithm}[H]
\small
\SetKwInput{Input}{Input}
\SetKwInput{Output}{Output}
\Input{$\mathbf x=(x_i)$ and  parameter $U$. }
\begin{description}[rightmargin=5mm]
\item[Step 1.]~Let $\xbf_+=(\max\{x_i,0\})$ and $\xbf_-=(\max\{-x_i,0\})$. Sort the components of the vector $\widetilde\xbf=[\xbf_+;\,\xbf_-]$ in a descending order, and let $(\widetilde x_{(i)})$ denote the sorted vector. Below, $z_{(i)}$ and $v_{(i)}$ follow the same indexing of components as $\widetilde x_{(i)}$.

\item[Step 2.]~Calculate $\zbf=(z_i)\in\R^{2d}$, for $i=1,\cdots, 2d$, by
\begin{align}
z_{i}=\begin{cases}
 {\widetilde x}_{i},&\text{if $\widetilde x_{i}\geq U$};
\\
0, &\text{otherwise}.
\end{cases}\nonumber
\end{align}

\item[Step 3.]~\textbf{If} $\mathbf 1^\top\zbf\leq R$,  set $\widetilde\vbf = \zbf$. 

~~~~~~\,\textbf{Else} compute $\widetilde\vbf=(\widetilde v_{i})$, for $i=1,\cdots,\ 2d$, by
\[  
\widetilde v_{(i)}= \begin{cases}
\widetilde x_{(i)}+\tau, &\text{if $i\leq \rho$};
\\
0,&\text{otherwise},
\end{cases}
\]
~~~~~~~~~~~~~~where $\tau= \frac{R-\sum_{i=1}^\rho  \widetilde x_{(i)}}{\rho}$ and 
$\rho=\max\left\{j:\, \widetilde x_{(j)}+\frac{R-\sum_{i=1}^{j} \widetilde  x_{(i)}}{j}\geq U\right\}$.
\end{description}
\Output{$\vbf=(\widetilde v_{i}:\,i=1,...,d)-(\widetilde v_{i}:\,i=d+1,...,2d)$.}
\caption{Per-iteration subroutine of SI-SGF.} \label{sub-alg}
\end{algorithm}


As mentioned,  Algorithm \ref{sub-alg}  equivalently solves a sparse projection problem, whose  exact formulation   is made explicit in the proposition below. Recall that, for Algorithm \ref{sub-alg}, $U$ is a user-specified parameter and $\xbf$ is the input.

Algorithm  \ref{sub-alg} involves $O(d\ln d)$-many arithmetic operations and thus is a reasonably efficient. In comparison, the randomized smoothing scheme in \eqref{randomized smoothing} yields at least $O(d)$-many arithmetic operations.

In the pseudo-codes above, we specify that Algorithm \ref{sub-alg} should run for $k=1,...,K$.  Yet, in implementation, the algorithm may terminate at the $(K-1)$-th iteration, because the  output of the algorithm relies only on results from the first $(K-1)$-many iterations.  The $K$-th iteration is used only for our subsequent theoretical analysis, which happens to involve $\xbf^{K+1}$.

\begin{proposition}\label{lemma1} 
Let $a,\,\lambda, \gamma$ be arbitrarily chosen positive scalars such that  $\gamma   \geq 2a$ and $a\lambda=U$. Let $\vbf$ denote the output of Algorithm \ref{sub-alg}. Then, we have:
\begin{itemize}
\item[a.] For all $i=1, \cdots, d$, either $\vert v_i\vert \geq U$ or $v_i=0$. 
\item[b.] Moreover, $\vbf$ is the optimal solution to the following optimization problem.
\begin{align}\label{lemma1-KKT-solution-3_solution}
  \min_{\vbf'\in\R^d}\,\left\{\frac{1}{2\gamma}\left\Vert \vbf'-\xbf \right\Vert^2 +\sum_{i=1}^{d} \frac{[a\lambda-\vert v_i\vert]_+}{a}\cdot \vert v_i'\vert:\, \Vert \vbf'\Vert_1 \leq R\right\}.
\end{align}
\end{itemize}
\end{proposition}
\begin{proof} See Appendix \ref{Proof of proposition 1}.\qed
\end{proof}

 \section{Main complexity results for the SI-SGF}\label{main results sec}
  In this section, we present our main complexity results  for the SI-SGF in solving \eqref{SP problem} when $F$ is convex   or strongly convex in Sections \ref{sec: convex} and \ref{sec: strongly convex}, respectively. In both cases, we prove that SI-SGF is dimension-insensitive.  Section \ref{sec alternative output} presents an alternative  output scheme, which is potentially more practical than the default random output in Algorithm \ref{main-alg}.  For all the proofs, we  mostly focus on the orders of complexity rates and  the constants involved may not have been optimized.
  

\subsection{Complexity of the SI-SGF in solving  convex S-ZOO problems}\label{sec: convex}
The complexity analysis relies on the following technical lemma, whose proof is postponed till Section \ref{sec proof lemma here} of the Appendix.
 \begin{lemma}\label{bounding one particular term}
  Suppose that Assumption \ref{mean and variance} holds. Let $(\xi^{m}:\,m=1,...,M)$ be a sample mini-batch of the random parameters $\xi$ in Problem \eqref{SP problem} and  $\ubf^m$ be a vector of iid symmetric Bernoulli random variables; that is, they are uniformly distributed random variables on $\{-1,\,1\}$.  For any $\xbf\in\R^d$, it holds that
 \begin{multline}
  \E_{\mathcal V_M}\left[\max_{\substack{\widehat S\in\{1,...,d\}:
  \\\vert \widehat S\vert\leq\frac{2R}{a\lambda}}}\left\Vert \sum_{m=1}^M\frac{f(\xbf+\delta \ubf^{m},\xi^{m})-f(\xbf,\xi^{m})}{M\delta}\ubf^{m}_{\widehat S} -\sum_{m=1}^M\frac{\nabla_{\widehat S}f(\xbf,\xi^{m})}{M}\right\Vert^2\right]   \\ \leq\,    \frac{{L^2\delta^2} d^2\cdot R}{a\lambda}+\frac{772  R\cdot \ln d}{a\lambda}\cdot  \frac{\sigma^2+\Vert\nabla F(\xbf)\Vert^2 }{M},\nonumber
\end{multline}
where $\mathcal V^M:=\left((\xi^m,\,\ubf^m):\,m=1,...,M\right)$.
 \end{lemma}
 \begin{proof}
See Section \ref{sec proof lemma here}.\qed
 \end{proof}
  
Let the parameters of Algorithms  \ref{main-alg} and  \ref{sub-alg} be set as follows.
\begin{equation}\label{parameter settings}
  \begin{aligned} &    U_k=U=a\lambda,
~\lambda=\frac{200L}{K\varpi},~a=\frac{1}{100L},~\gamma_k=\gamma=\frac{1}{50L},~\text{for}~k=0,...,K,\\
&  \delta\leq  \frac{\theta}{Kd},~~~
M=\left\lceil\frac{50K^2\varpi\max\{1,\,\sigma^2\}}{L^2}\cdot \ln d\right\rceil,
  \end{aligned}    
\end{equation}
where $\theta>0$ and $\varpi>0$  are some user-specified hyper-parameters.
Now we are ready to present the main result for convex S-ZOO problems.  



  \vspace{-5pt}
 \begin{theorem}\label{main theorem}
Suppose that Assumptions \ref{sample simulation} through \ref{bounded} hold.   Given that the hyper-parameters are set as in \eqref{parameter settings}  and that $K\geq  30 L^2 R$, there exists  a constant $C_1>0$ such that
the output solution of Algorithm \ref{main-alg} satisfies
 \begin{multline}
  \mathbf E\left[F(\xbf^Y)-F(\xbf^*)\right]\leq\, \frac{C_1L  \Vert \xbf^1-\xbf^*\Vert^2 }{K}+\frac{C_1 LR}{K}\cdot (1+\varpi^{-1}+\varpi\theta^2)
 \\  +\frac{C_1  L}{K^2}\cdot \left(\frac{\theta^2}{d}+\frac{\varpi^{-1}}{\ln d}\right),\label{error bound in convex}
\end{multline}
where $\mathbf E$ is the expectation taken over all the random variables in Algorithm \ref{main-alg}
 \end{theorem}
\begin{proof}
Firstly, for some $a$ and $\lambda$ such that $a\cdot \lambda =U$ and $a\leq \frac{\gamma_k}{2}$,
Proposition \ref{lemma1} (therein with $\xbf:=\xbf^k-\gamma_k\mathbf g_k^{\delta}(\xbf^k)$, $\gamma:=\gamma_k$ and  $\vbf:=\xbf^{k+1}=(x_{i}^{k+1})$) implies that
Algorithm \ref{sub-alg} computes  an  optimal solution $\xbf^{k+1}$ to the following optimization problem:
\begin{align}\nonumber
  \min_{\vbf'\in\R^d}\,\left\{\frac{1}{2\gamma_k}\left\Vert \vbf'-\xbf^k+\gamma_k\mathbf g_k^{\delta}(\xbf^k) \right\Vert^2 +\sum_{i=1}^{d} \frac{[a\lambda-\vert v_i\vert]_+}{a} \vert v_i'\vert:\, \Vert \vbf'\Vert_1 \leq R\right\}.
\end{align}
The first-order necessary optimality conditions of the problem above yield that, for all $\xbf:\,\Vert\xbf\Vert_1\leq R$:
\begin{equation}\label{first-order condition 1}
  \left\langle \frac{1}{\gamma_k} \left(\xbf^{k+1}-\xbf^k+\gamma_k\mathbf g_k^{\delta}(\xbf^k)\right) +\boldsymbol\varrho^{k+1},\,\xbf-\xbf^{k+1}\right\rangle\geq 0,
\end{equation}
where $\boldsymbol\varrho^{k+1}=\left(\frac{[a\lambda-\vert x_{i}^{k+1}\vert]_+}{a} \cdot \Gamma_{\vert x_{i}^{k+1}\vert}:\,i=1,...,d\right)$ and $\Gamma_{\vert x_{i}^{k+1}\vert}$ is a subgradient of the absolute value function $\vert\cdot\vert$ at $x_{i}^{k+1}$. If we plug in  $\xbf:=\xbf^*\in \arg\min_{\xbf\in\R^d}\,F(\xbf)\cap\{\xbf:\,\Vert \xbf\Vert_1\leq R\}$ and invoke the convexity of $\vert\cdot\vert$, we  have
\[
  \left\langle \boldsymbol\varrho^{k+1},\,\xbf^*-\xbf^{k+1}\right\rangle
\leq \sum_{i=1}^d \frac{[a\lambda-\vert x_{i}^{k+1}\vert]_+}{a}(\vert x^*_i\vert-\vert x_{i}^{k+1}\vert). \]
Let $\lambda_i:=\frac{[a\lambda-\vert x_{i}^{k+1}\vert]_+}{a}$. As per  Lemma \ref{lemma1}, we have    $ \lambda_i=\frac{[a\lambda-\vert x_{i}^{k+1}\vert]_+}{a}=0$ for all $i$ such that $x_i^{k+1}\neq 0$ and $\lambda_i= \frac{[a\lambda-\vert x_{i}^{k+1}\vert]_+}{a}=\lambda$ for all $i$ such that $x_i^{k+1}=0$. 
Therefore, we may continue from \eqref{first-order condition 1} above (with $\xbf:=\xbf^*$ therein) to obtain   
\begin{multline}
\Vert \xbf^{k+1}\Vert^2-(\xbf^{k+1})^\top\xbf^k+(\xbf^{*})^\top(\xbf^{k}-\xbf^{k+1})
\\  \leq \left\langle \gamma_k \gbf_k^\delta(\xbf^k),\, \xbf^*-\xbf^{k+1}\right\rangle+\gamma_k\sum_{i=1}^d\lambda_i(\vert x^*_i\vert-\vert x_i^{k+1}\vert).\label{to be useful}\end{multline}
Notice that
\begin{align}
  &\nonumber\frac{1}{2}\Vert \xbf^{k+1}-\xbf^*\Vert^2- \frac{1}{2}\Vert \xbf^k-\xbf^*\Vert^2
  \\  =\,&   \Vert \xbf^{k+1}\Vert^2+\left\langle \xbf^{k} - \xbf^{k+1},\,\xbf^*\right\rangle-\frac{1}{2}\Vert \xbf^{k}-\xbf^{k+1} \Vert^2-(\xbf^{k+1})^\top\xbf^{k}.\nonumber
  \end{align}
Let $S_k:=\{i:\,x_i^{k}-x_i^{k+1}\neq 0\}$,  $\mathbf g(\xbf^k): = \frac{1}{M} \sum_{m=1}^M\nabla f(\xbf^k,\xi^{k,m})$, and, thus, $\mathbf g_{S_k}(\xbf^k) = \frac{1}{M} \sum_{m=1}^M\nabla_{S_k} f(\xbf^k,\xi^{k,m})$.  By Lemma \ref{lemma1},  we know that $\Vert \xbf^k\Vert_0\leq \frac{R}{a\lambda}$ and $\Vert \xbf^{k+1}\Vert_0\leq \frac{R}{a\lambda}$. This comes immediately from the observation that, due to   Lemma \ref{lemma1},  $\vert x_i^k\vert \geq U$ if $x_i^k\neq 0$ (and  $\vert x_i^{k+1}\vert \geq U$ if $x_i^{k+1}\neq 0$), as well as  $\Vert \xbf^k\Vert_1\leq R$ (and $\Vert \xbf^{k+1}\Vert_1\leq R$, respectively). Consequently, $\vert S_k\vert\leq  \frac{2 R}{U}= \frac{2 R}{a\lambda}$.

 In view of \eqref{to be useful}, we obtain from the above that
  \begin{align}
  \nonumber &  \frac{1}{2}\Vert \xbf^{k+1}-\xbf^*\Vert^2- \frac{1}{2}\Vert \xbf^k-\xbf^*\Vert^2
 \\ \nonumber\leq\,&  \langle \gamma_k\mathbf g_k^{\delta}(\xbf^k),\,\xbf^*-\xbf^{k+1}\rangle- \frac{1}{2}\Vert \xbf^{k+1}-\xbf^{k}\Vert^2 +\gamma_k\sum_{i=1}^d\lambda_i (\vert x^*_i\vert-\vert x_i^{k+1}\vert) 
\\
 \nonumber \leq\,&   \langle \gamma_k\mathbf g_k^{\delta}(\xbf^k),\,\xbf^*-\xbf^{k}\rangle+\langle \gamma_k\mathbf g_k^{\delta}(\xbf^k),\,\xbf^{k}-\xbf^{k+1}\rangle 
- \frac{1}{2}\Vert \xbf^{k+1}-\xbf^{k}\Vert^2\nonumber
 \\&~~~~~+\gamma_k\sum_{i=1}^d\lambda_i\left(\vert x_i^*\vert -\vert x_i^{k+1}\vert\right) \nonumber
\\\label{3rd last inequality to explain} 
= &  \langle \gamma_k\mathbf g_k^{\delta}(\xbf^k),\,\xbf^*-\xbf^{k}\rangle+\langle \gamma_k\mathbf g^{\delta}_{S_k}(\xbf^k),\,\xbf^{k}_{S_k}-\xbf^{k+1}_{S_k}\rangle 
- \frac{1}{2}\Vert \xbf^{k+1}-\xbf^{k}\Vert^2  
 \\&~~~~~+\gamma_k\sum_{i=1}^d\lambda_i\left(\vert x_i^*\vert -\vert x_i^{k+1}\vert\right) \nonumber
 \\
 \label{2nd last inequality to explain} 
\leq\,&   \langle \gamma_k\mathbf g_k^{\delta}(\xbf^k),\,\xbf^*-\xbf^{k}\rangle+\frac{\gamma_k^2}{2\eta} \Vert \mathbf g^{\delta}_{S_k}(\xbf^k)\Vert^2  +\frac{\eta}{2}\Vert\xbf_{S_k}^{k+1}-\xbf_{S_k}^{k}\Vert^2 
\\&  - \frac{1}{2}\Vert \xbf^{k+1}-\xbf^{k}\Vert^2
+\gamma_k\sum_{i=1}^d\lambda_i\left(\vert x_i^*\vert -\vert x_i^{k+1}\vert\right) \nonumber
\\ \nonumber \leq\,&   \langle \gamma_k\mathbf g_k^{\delta}(\xbf^k),\,\xbf^*-\xbf^{k}\rangle
+\frac{\eta}{2}\Vert\xbf^{k+1}-\xbf^{k}\Vert^2 
- \frac{1}{2}\Vert \xbf^{k+1}-\xbf^{k}\Vert^2 \nonumber
\\
&  +\frac{\gamma_k^2}{2\eta} \Vert \mathbf g^{\delta}_{S_k}(\xbf^k)-\mathbf g_{S_k}(\xbf^k)+\mathbf g_{S_k}(\xbf^k)-\nabla_{S_k} F(\xbf^k)+\nabla_{S_k} F(\xbf^k)\Vert^2\nonumber
\\
&  +\gamma_k\sum_{i=1}^d\lambda_i\left(\vert x_i^*\vert -\vert x_i^{k+1}\vert\right),
\label{inequality to use direction results}
\end{align}
 where the \eqref{3rd last inequality to explain} is by the definition of $\gbf_{S_k}(\xbf^k)$ and \eqref{2nd last inequality to explain} is due to $\frac{\Vert\mathbf a\Vert^2 +\Vert\mathbf b\Vert^2}{2}\geq \langle \mathbf a,\,\mathbf b\rangle$ for arbitrary vectors $\mathbf a,\,\mathbf b\in\R^d$.
 As we set $\gamma_k = \gamma$, and $\eta = 1$,  we can continue from the above to obtain
\begin{align}
&\frac{1}{2}\Vert \xbf^{k+1}-\xbf^*\Vert^2- \frac{1}{2}\Vert \xbf^k-\xbf^*\Vert^2\nonumber
\\
\nonumber \leq\,&   \langle \gamma\mathbf g_k^{\delta}(\xbf^k),\,\xbf^*-\xbf^{k}\rangle 
+\gamma\sum_{i=1}^d\lambda_i\vert x_i^*\vert  +\frac{3\gamma^2}{2} \Vert  \nabla F(\xbf^k)\Vert^2\nonumber
\\
&  +\frac{3\gamma^2}{2} \Vert \mathbf g^{\delta}_{S_k}(\xbf^k)-\mathbf g_{S_k}(\xbf^k)\Vert^2+\frac{3\gamma^2}{2} \Vert \mathbf g (\xbf^k)-\nabla  F(\xbf^k)\Vert^2.\label{to start with}
 \end{align}

By Part (a) of Lemma \ref{lemma approx delta}, we may immediately obtain $\vert F^\delta(\xbf)-F(\xbf)\vert\leq \frac{L}{2}d\delta^2$. Further invoking Eq.\ \eqref{convexity inequality here}, and the convexity of $F$,  we know that 
 \begin{align}  & \E_{\left((\xi^{k,m},\ubf^{k,m}):\,m=1,...,M\right)}\left[\langle  \mathbf g_k^{\delta}(\xbf^k),\,\xbf^*-\xbf^{k}\rangle\right]\nonumber
 \\=&\E_{(\ubf^{k,m}:\,m=1,...,M)}\left[M^{-1}\sum_{m=1}^M\left\langle \frac{F(\xbf^k+\mu\cdot \ubf^{k,m})-F(\xbf^k)}{\delta}\ubf^{k,m}\,\xbf^*-\xbf^{k}\right\rangle\right]\nonumber\\
 \leq&\E_{(\ubf^{k,m}:\,m=1,...,M)}\left[M^{-1}\sum_{m=1}^M \left\langle \nabla F(\xbf^k+\mu\cdot \ubf^{k,m}),\ubf^{k,m}\right\rangle \cdot \left\langle \ubf^{k,m}, \,\xbf^*-\xbf^{k}\right\rangle\right]\nonumber\\
 =&\left\langle   \nabla F^\delta(\xbf^k),\,\xbf^*-\xbf^{k}\right\rangle \leq F^\delta(\xbf^*)-F^\delta(\xbf^k)\leq F(\xbf^*)-F(\xbf^k)+\frac{L}{2}\delta^2d.\label{convex expect here}
 \end{align}
 Meanwhile, we can obtain an upper bound on $\Vert \mathbf g^{\delta}_{S_k}(\xbf^k)-\mathbf g_{S_k}(\xbf^k)\Vert^2$ in \eqref{to start with} by invoking Lemma \ref{bounding one particular term}.  More specifically, we have
\begin{align} 
  &\E_{\left((\xi^{k,m},\,\ubf^{k,m}):\,m=1,...,M\right)}\left[\Vert \mathbf g^{\delta}_{S_k}(\xbf^k)-\mathbf g_{S_k}(\xbf^k)\Vert^2\right]\nonumber
\\ \stackrel{\vert S_k\vert\leq \frac{2R}{a\lambda}}{\leq} ~~&\E_{\left((\xi^{k,m},\,\ubf^{k,m}):\,m=1,...,M\right)}\left[\max_{\widehat S\in\{1,...,d\}:\vert \widehat S\vert\leq\frac{2R}{a\lambda}}\left.\Vert \mathbf g^{\delta}_{\widehat S}(\xbf^{k}) - \gbf_{\widehat S}(\xbf^k)\Vert^2\right.\right]   \nonumber
  \\ \stackrel{\text{Lemma \ref{bounding one particular term}}}{\leq}~~&   \frac{{L^2\delta^2} d^2 R}{a\lambda}+\frac{772  R\cdot \ln d}{a\lambda}\cdot  \frac{\sigma^2+\Vert\nabla F(\xbf^k)\Vert^2 }{M}.\label{second term}
\end{align}
Combining the above with \eqref{to start with}, \eqref{convex expect here}, and Assumption \ref{mean and variance}, and taking expectation with respect to $\mathcal W =((\xi^{k,m},\ubf^{k,m}):\, k=1,...,K,\,m=1,...,M)$, we obtain
\begin{align}
&  \nonumber\frac{1}{2}\E_{\mathcal W}\left[\Vert \xbf^{k+1}-\xbf^*\Vert^2\right]- \frac{1}{2}\E_{\mathcal W}\left[\Vert \xbf^k-\xbf^*\Vert^2\right]\\
  \nonumber \leq\,&  \E_{\mathcal W}[\gamma   F(\xbf^*)]-\E_{\mathcal W}[\gamma F(\xbf^k)]+\frac{\gamma L \delta^2d}{2}
+\gamma\lambda\Vert\xbf^*\Vert_1 +\frac{3\gamma^2}{2} \E_{\mathcal W}\left[\Vert  \nabla F(\xbf^k)\Vert^2\right]\nonumber
\\
&  +\frac{3\gamma^2}{2}   \left( \frac{{L^2\delta^2} d^2R}{a\lambda}+\frac{772  R\cdot \ln d}{a\lambda}\cdot  \frac{\sigma^2+\Vert\nabla F(\xbf^k)\Vert^2 }{M}\right)+\frac{3\gamma^2\sigma^2}{2 M}.
\label{to start with 2}
 \end{align} 
 
By the well-known inequality for convex and smooth function (with $L$-Lipschitz gradient), we have  $F(\xbf)-F(\xbf^*)-\langle \nabla F(\xbf^*),\xbf-\xbf^*\rangle\geq \frac{1}{2L}\Vert \nabla F(\xbf)-\nabla F(\xbf^*)\Vert^2$  for any $\xbf\in\R^d$.
As $\nabla F(\xbf^*)=\0$, we then have $F(\xbf^k)-F(\xbf^*)\geq \frac{1}{2L} \Vert \nabla F(\xbf^k)-\nabla F(\xbf^*)\Vert^2=\frac{1}{2L}\Vert\nabla F(\xbf^k)\Vert^2$. This, combined with \eqref{to start with 2}, leads to
\[  
 \begin{aligned}
  &  \frac{1}{2}\E_{\mathcal W}\left[\Vert \xbf^{k+1}-\xbf^*\Vert^2\right]- \frac{1}{2}\E_{\mathcal W}\left[\Vert \xbf^k-\xbf^*\Vert^2\right]\\
  \leq\,&  \left(\gamma-3L\gamma^2-\frac{2316 LR \gamma^2\ln d}{a\lambda M}\right)\left( F(\xbf^*)-\E_{\mathcal W}[F(\xbf^k)]\right)+\frac{\gamma L\delta^2 d}{2}
+\gamma\lambda\Vert\xbf^*\Vert_1
\\
&  +\frac{3\gamma^2}{2a\lambda}  {L^2\delta^2} d^2 R +\frac{1158 R \gamma^2\sigma^2}{a\lambda M}\ln d+\frac{3\gamma^2\sigma^2}{2 M}.
\label{to start with 3}
 \end{aligned}
 \]
Applying the above inequality recursively for all $k=1,...,K$, and summing them up, we obtain
\[  
\begin{aligned}
&   \sum_{k=1}^K\left[\frac{1}{2}\E_{\mathcal W}\left[\Vert \xbf^{k+1}-\xbf^*\Vert^2\right]- \frac{1}{2}\E_{\mathcal W}\left[\Vert \xbf^k-\xbf^*\Vert^2\right]\right]\\
\leq\,&  \sum_{k=1}^K\left(\gamma-3\gamma^2L-\frac{2316 LR\gamma^2\ln d}{a\lambda M}\right)\left( F(\xbf^*)-\E_{\mathcal W}[F(\xbf^k)]\right)+ \frac{K\gamma L\delta^2 d}{2}\\ 
&  +K\gamma \lambda\Vert\xbf^*\Vert_1 + \frac{3K\gamma^2}{2a\lambda} {L^2\delta^2} d^2 R +\frac{1158 R \gamma^2\sigma^2K}{a\lambda M}\ln d+\frac{3\gamma^2\sigma^2K}{2 M}.
\end{aligned}
\]
 By some simplification and the definition of $\mathbf E$, which is the expectation over all the random variables in Algorithm \ref{main-alg}, we have

\begin{align}
&  \frac{1}{2}\E_{\mathcal W}\left[\Vert \xbf^{K+1}-\xbf^*\Vert^2\right]- \frac{1}{2}\E_{\mathcal W}\left[\Vert \xbf^1-\xbf^*\Vert^2\right]\\
 &   \leq\,K\left(\gamma-3\gamma^2L-\frac{2316 L R\gamma^2\ln d}{a\lambda M}\right)\left( F(\xbf^*)-\mathbf E[F(\xbf^Y)]\right) +K\gamma \lambda\Vert\xbf^*\Vert_1\nonumber
\\  &  \quad  +\frac{3K\gamma^2}{2 a\lambda}  {L^2\delta^2} d^2 R  +  \frac{K\gamma L \delta^2\nonumber d}{2}+\frac{1158 R \gamma^2\sigma^2K}{a\lambda M}\ln d+\frac{3\gamma^2\sigma^2K}{2 M}.
\end{align}
 By properly choosing parameters, to be elaborated later, we can ensure that $\alpha:=\gamma-3\gamma^2L-\frac{2316\cdot R\cdot \gamma^2\cdot \ln d}{a\lambda\cdot M}>0$. After some rearrangement, we obtain
\begin{multline}
 \mathbf E\left[ F(\xbf^Y)-F(\xbf^*)\right]
\leq\,  \frac{\E_{\mathcal W}\left[\Vert \xbf^1-\xbf^*\Vert^2\right]}{2 K\alpha}+\frac{\lambda\gamma\Vert\xbf^*\Vert_1}{\alpha}
\\+\frac{3\gamma^2}{2\alpha a\lambda} {L^2\delta^2} d^2 R+ 
\frac{\gamma L \delta^2 d}{2\alpha}   + \frac{1158 R \gamma^2\sigma^2}{a\alpha\lambda M}\ln d+\frac{3\gamma^2\sigma^2}{2\alpha M}.\label{final theorem 1}
\end{multline}
Let  $\lambda= \frac{200 L }{K \varpi  }$, $\gamma=\frac{1}{50L}$, $a=\frac{\gamma}{2}=\frac{1}{100L}$,   $M=\left\lceil\frac{50 K^2 \varpi \max\{1,\,\sigma^2\}}{L^2}\ln d\right\rceil$, and $K\geq   L^2 R$. Thus, $1-3\gamma L-\frac{2316 L R\gamma \ln d}{a\lambda M}\geq 1-\frac{3}{50}-\frac{2316}{5000} = 0.4768$ and $\alpha=\gamma-3\gamma^2 L-\frac{96 R\gamma^2 L }{a\lambda M}\geq \frac{1}{105L}$. 
We obtain the desired result by plugging the above into \eqref{final theorem 1}, while recalling that $\delta\leq \frac{\theta}{Kd }$ and $\xbf^1$ is deterministic.\qed
\end{proof}

 \begin{remark}We would like to make a few remarks on the above result.
\begin{itemize}
\item The  algorithm parameters can be chosen with more flexibility to achieve the promised query complexity than what is given in \eqref{parameter settings}.  In fact, if we set any of $\lambda$, $a$, $\gamma$, $\delta$, or $M$ to be some constant multiple of the current value, the same complexity rate can be achieved.
    \item To obtain an $\epsilon$-suboptimal solution, \eqref{error bound in convex} indicates an iteration complexity of
$\mathcal O\left( \frac{  (D_0+R)L}{\epsilon}\right),$ where $D_0:=\Vert\xbf^1-\xbf^*\Vert^2$.
  Since each  per-iteration subroutine invokes   $ \mathcal O\left( \frac{(D_0+  R)^2\sigma^2\cdot \ln d}{\epsilon^2}\right)$-many  queries of the stochastic zeroth-order oracle,
 the total number of calls to the  oracle is 
  \begin{align}  \mathcal O\left( \frac{ \left(D_0 +  R\right)^3L\sigma^2\cdot \ln d}{\epsilon^3 }\right),\label{oracle complexity detailed}
  \end{align}
   which is dimension-free up to a logarithmic term, when $R$, $\epsilon$, $L$, and $\sigma$ are fixed.
   
\item By \eqref{oracle complexity detailed}, we know that the proposed SI-SGF algorithm tends to be more effective when  $R$ is small. In particular, when $\xbf^1=\0$ and there exists an $s$-sparse solution (formalized in Assumption \ref{sparsity assumption}) for some $s$ such that $\,1\leq s\ll d$, then \eqref{oracle complexity detailed}  immediately reduces to
  \begin{align}  \mathcal O\left( \frac{s^3L\sigma^2\ln d}{\epsilon^3 }\right).\label{oracle complexity}
  \end{align}
As a benchmark, the iteration complexity of the randomized stochastic gradient free (RSGF) algorithm  for zeroth-order optimization in \cite{ghadimi2013stochastic} is
\begin{align}
  \mathcal O\left(\frac{d  D_0\sigma^2}{\epsilon^2}\right)=\begin{cases}\mathcal O\left(\frac{d^2  \sigma^2}{\epsilon^2}\right)&\text{In general};\\
\mathcal O\left(\frac{sd  \sigma^2}{\epsilon^2}\right)&\text{If solution $\xbf^*$ is $s$-sparse}.
\end{cases}\label{traditional complexity}
\end{align}
Thus, if $d\gg\frac{s^2}{\epsilon}$, the rate in \eqref{oracle complexity} is significantly more appealing than \eqref{traditional complexity}.
\item As  in Table \ref{table one}, compared to the state-of-the-art algorithm for high-dimensional S-ZOO in \cite{wang2018stochastic,cai2020zeroth,balasubramanian2018zeroth,balasubramanian2018zeroth_neurips},  the proposed algorithm does not rely on any assumption of sparse gradient, compressible gradient, or additive randomness. Instead, we only require that the optimal solution is (approximately) sparse; that is $\Vert \xbf^*\Vert_1\leq R$ for some small $R$, and $R$ indeed can be small when $\xbf^*$ is sparse. To our knowledge, Theorem \ref{main theorem} is the first result for dimension-insensitive S-ZOO under the relatively weak assumption of (weakly) sparse optimal solution.
\item

Under the same  assumptions and  parameter settings as in Theorem \ref{main theorem}, by Markov's inequality,  we  further obtain that $\forall \epsilon>0$, there exists a constant $C_1>0$ such that
\begin{multline}
  \text{Prob}\left[ F(\xbf^Y) - F(\xbf^*)\leq \varepsilon\right]
   \geq 1- \frac{C_1L  \Vert \xbf^1-\xbf^*\Vert^2 }{K\varepsilon}
   \\-\frac{C_1 LR}{K\varepsilon}\cdot (1+\varpi^{-1}+\varpi\theta^2)-\frac{C_1  L}{K^2\varepsilon}\cdot \left(\frac{\theta^2}{d}+\frac{\varpi^{-1}}{\ln d}\right).\label{markov bound}
\end{multline}

\item The implementation of the algorithm does not rely on the knowledge of the true sparsity-level $s$ of an optimal solution. Instead, an over-estimate of its $\ell_1$-norm will suffice to set the hyper-parameters of Algorithm \ref{main-alg}. 
\item Finally, the effectiveness of the proposed algorithm depends on machine precision $\hat\epsilon$, the  relative approximation error due to rounding in floating point arithmetic. In particular, it is implicitly required that, to implement the SI-SGF, the quantity $\delta$ cannot be smaller than $\hat\epsilon$. In other words, it is required that $\frac{\theta}{Kd}\geq\delta> \hat\epsilon$. For the double precision on a 32-bit computer,   $\hat\epsilon=2^{-52}\approx 10^{-16}$, which requires that $Kd< \frac{\theta}{\hat\epsilon}\approx 10^{16}\cdot \theta$. Thus, in spite of the worst-case dimension-insensitive complexity of the proposed SI-SGF, there is an upper limit on the admissible problem dimensionality. This limit is less stringent when the machine precision improves.

\end{itemize}
\end{remark}
%
%
%

\subsection{Complexity of the SI-SGF in solving strongly convex S-ZOO problems}\label{sec: strongly convex}
We now consider  solving a strongly convex S-ZOO problem. The assumptions on the strong convexity of $F(\cdot)$ and the sparsity of an optimal solution are formalized in Assumptions \ref{strongly convex assumption} and \ref{sparsity assumption}, respectively. Before presenting our main result for strongly convex S-ZOO problems, we recall the assumption that $R\geq 1$.  
 
 \begin{theorem}\label{main theorem strongly convex}
Suppose that Assumptions \ref{sample simulation} through \ref{sparsity assumption} hold,
and the hyper-parameters in Algorithm \ref{main-alg} are set as follows:
\begin{equation}
\begin{aligned}
&\gamma_k=\frac{2}{\mu\cdot (k+\left\lceil \frac{100L}{\mu\varpi}\right\rceil+1)}, ~ a_k=\frac{\gamma_{k-1}}{2},~ U_k=a_k\cdot\lambda,~\text{for}~k=1,...,K;
\\ 
&  \delta=\frac{\theta}{ {K^{1.5}   d}},~~ \lambda= \frac{200L}{K\varpi}, ~~M=\left\lceil 50 K^{3} \varpi
\max\{1,\,\sigma^2\}\cdot \mu\cdot L^{-3}\cdot \ln d\right\rceil,
\end{aligned}\label{parameters strongly convex}
\end{equation}
where $ K\geq    \frac{1}{\sqrt{\mu}} L^{3/2} R^{1/2} $ with $\varpi>0$ and $\theta>0$ being some user-specified hyper-parameters.
Then, the output of Algorithm \ref{main-alg} satisfies that
\begin{multline}
   \mathbf E\left[ F(\xbf^Y)- F(\xbf^*)\right]
     \\
 \leq\,   \,      C_2\cdot \frac{L^2}{K^2\mu}\left[s+R \cdot (1+\theta^2\varpi\cdot \mu)+ \Vert \xbf^1-\xbf^*\Vert^2  + \frac{\mu \theta^2}{KLd}+\frac{1}{\varpi^2 K}\right],\label{error bound strongly convex}
\end{multline}
 for some  constant $C_2>0$. 
 \end{theorem}
\begin{proof}
By the same argument as in deriving \eqref{inequality to use direction results}, we obtain
\begin{align}
  &  \nonumber\frac{1}{2}\Vert \xbf^{k+1}-\xbf^*\Vert^2- \frac{1}{2}\Vert \xbf^k-\xbf^*\Vert^2
  \\
  \nonumber \leq\,&   \langle \gamma_k\mathbf g_k^{\delta}(\xbf^k),\,\xbf^*-\xbf^{k}\rangle
+\frac{\eta-1}{2}\Vert\xbf^{k+1}-\xbf^{k}\Vert^2  
\\
&  +\frac{\gamma_k^2}{2\eta} \Vert \mathbf g^{\delta}_{S_k}(\xbf^k)-\mathbf g_{S_k}(\xbf^k)+\mathbf g_{S_k}(\xbf^k)-\nabla_{S_k} F(\xbf^k)+\nabla_{S_k} F(\xbf^k)\Vert^2.\nonumber
\\&+\gamma_k\sum_{i=1}^d\lambda_i\vert x_i^*\vert -\gamma_k\sum_{i=1}^d\lambda_i\vert x_i^{k+1}\vert.  \nonumber
\end{align}
As $0\leq \lambda_i\leq \lambda$ for all $i$, we have $ \gamma_k\sum_{i=1}^d\lambda_i(\vert x_i^*\vert - \vert x_i^{k+1}\vert) \leq \gamma_k \sum_{i:\,x_i^*\neq 0}\lambda_i\vert x_i^{k+1}-x_i^*\vert\leq \gamma_k\lambda\sqrt{s} \Vert \xbf^{k+1}-\xbf^*\Vert.$
Let $S_k:=\{i:\,x_i^k-x_i^{k+1}\neq 0\}$. Then, 
\begin{align}
&  \nonumber\frac{1}{2}\Vert \xbf^{k+1}-\xbf^*\Vert^2- \frac{1}{2}\Vert \xbf^k-\xbf^*\Vert^2
\\
  \nonumber \leq\,&   \langle \gamma_k\mathbf g_k^{\delta}(\xbf^k),\,\xbf^*-\xbf^{k}\rangle
+\frac{\eta-1}{2}\Vert\xbf^{k+1}-\xbf^{k}\Vert^2 
+\gamma_k\lambda\sqrt{s} \Vert \xbf^{k+1}-\xbf^*\Vert \nonumber
\\
+&\frac{3\gamma_k^2}{2\eta} \Vert \mathbf g^{\delta}_{S_k}(\xbf^k)-\mathbf g_{S_k}(\xbf^k)\Vert^2+\frac{3\gamma_k^2}{2\eta}\Vert\mathbf g(\xbf^k)-\nabla F(\xbf^k)\Vert^2 +\frac{3\gamma_k^2}{2\eta}\Vert\nabla F(\xbf^k)\Vert^2.\nonumber
\end{align}
We claim that $|S_k|\leq \frac{2 R}{a_{k+1}\lambda}$. To see this, we observe that $\Vert \xbf^k\Vert_0\leq \frac{R}{a_{k}\lambda}$ and $\Vert \xbf^{k+1}\Vert_0\leq \frac{R}{a_{k+1}\lambda}$. Actually, by Proposition \ref{lemma1},  $\vert x_i^k\vert \geq U_k=a_k\lambda \geq a_{k+1}\lambda= U_{k+1}$ if $x_i^k\neq 0$ (and  $\vert x_i^{k+1}\vert \geq U_{k+1}$ if $x_i^{k+1}\neq 0$), as well as  $\Vert \xbf^k\Vert_1\leq R$ (and $\Vert \xbf^{k+1}\Vert_1\leq R$). Consequently, $\vert S_k\vert\leq \frac{R}{U_k}+\frac{R}{U_{k+1}}\leq  \frac{2 R}{U_{k+1}}= \frac{2 R}{a_{k+1}\lambda}$.
 
Consider Lemma \ref{bounding one particular term}, as in deriving \eqref{second term} (where we let $a$ therein to be $a_{k+1}$). We  then  have $\E_{\mathcal U^k}[\Vert \mathbf g^{\delta}_{S_k}(\xbf^k)-\mathbf g_{S_k}(\xbf^k)\Vert^2]\leq \frac{{L^2\delta^2} d^2 R}{a_{k+1}\lambda}+\frac{772  R\cdot \ln d}{a_{k+1}\lambda}\cdot  \frac{\sigma^2+\Vert\nabla F(\xbf^k)\Vert^2 }{M}$, where $\mathcal U^k:=((\ubf^{k,m},\xi^{k,m}):\,m=1,...,M)$. Meanwhile, observe that   $\E_{\mathcal U^k}\left[\Vert \gbf_{S_k}(\xbf^k)-\nabla_{S_k} F(\xbf^k)\Vert^2\right]\leq \frac{\sigma^2}{M}$. We may then continue  to obtain
\begin{align}
 &  \nonumber\frac{1}{2} \E_{\mathcal U^k}\left[\Vert \xbf^{k+1}-\xbf^*\Vert^2\right]- \frac{1}{2}  \E_{\mathcal U^k}\left[ \Vert \xbf^k-\xbf^*\Vert^2 \right]
 \\
  \nonumber \leq\,&   \E_{\mathcal U^k}\left[\langle \gamma_k\mathbf g_k^{\delta}(\xbf^k),\,\xbf^*-\xbf^{k}\rangle\right]
+\frac{\eta-1}{2}\E_{\mathcal U^k}\left[\Vert\xbf^{k+1}-\xbf^{k}\Vert^2 \right]
\nonumber
\\
&  +\frac{3\gamma_k^2}{2\eta}\left(\frac{{L^2\delta^2} d^2 R}{a_{k+1}\lambda}+\frac{772  R\cdot \ln d}{a_{k+1}\lambda}\cdot  \frac{\sigma^2+\Vert\nabla F(\xbf^k)\Vert^2 }{M}\right)+\frac{3\gamma_k^2\sigma^2}{2\eta M}
\nonumber
\\&  +\frac{3\gamma_k^2}{2\eta}\E_{\mathcal U^k}\left[\Vert\nabla  F(\xbf^k)\Vert^2\right]+\gamma_k\lambda\sqrt{s}\cdot\E_{\mathcal U^k} \Vert \xbf^{k+1}-\xbf^*\Vert \nonumber
 \\ \label{inequality here to use further} \leq\,&   \E_{\mathcal U^k}\left[\langle \gamma_k\mathbf g_k^{\delta}(\xbf^k),\,\xbf^*-\xbf^{k}\rangle\right]
+\frac{\eta-1}{2}\E_{\mathcal U^k}\left[\Vert\xbf^{k+1}-\xbf^{k}\Vert^2 \right]
\\
&+\frac{3\gamma_k^2}{2\eta}\left(\frac{{L^2\delta^2} d^2 R}{a_{k+1}\lambda}+\frac{772  R\cdot \ln d}{a_{k+1}\lambda}\cdot  \frac{\sigma^2 }{M}\right) \nonumber
\\&+\frac{3\gamma_k^2}{2\eta} \left(1+\frac{772R\cdot  \ln d}{a_{k+1} \lambda M}\right)\E_{\mathcal U^k}\left[\Vert\nabla  F(\xbf^k)\Vert^2\right]+\gamma_k\lambda\sqrt{s}\cdot\E_{\mathcal U^k} \Vert \xbf^{k+1}-\xbf^*\Vert.\nonumber
\end{align}

 Similar to \eqref{convex expect here} and in view of Part (b) of Lemma \ref{lemma approx delta} and the strong convexity of $F$ (with modulus $\mu$),  we know that 
 \begin{align}
 &\E_{\left((\xi^{k,m},\ubf^{k,m}):\,m=1,...,M\right)}\left[\langle  \mathbf g_k^{\delta}(\xbf^k),\,\xbf^*-\xbf^{k}\rangle\right]  
\leq  \left\langle   \nabla F^{\delta}(\xbf^k),\,\xbf^*-\xbf^{k}\right\rangle \nonumber
  \\\stackrel{\substack{\text{Strong convexity}\\\text{ and \eqref{test new results here useful}}}}{\leq}& F^\delta(\xbf^*)- F^\delta(\xbf^k) -\frac{\mu}{2} \Vert \xbf^*-\xbf^k\Vert^2 \nonumber
 \\\stackrel{\substack{\text{Lemma \ref{lemma approx delta}}\\
 \text{and \eqref{convexity inequality here}}}}{\leq}~~~&F (\xbf^*)- F (\xbf^k) -\frac{\mu}{2} \Vert \xbf^*-\xbf^k\Vert^2 +\frac{1}{2}\delta^2Ld.\nonumber 
 \end{align}
  Since $F$ is convex and $\nabla F$ is Lipschitz continuous, we have 
$F(\xbf^k)-F(\xbf^*)-\langle \nabla F(\xbf^*),\,\xbf^k-\xbf^*\rangle\geq \frac{1}{2L}\Vert \nabla F(\xbf^*)-\nabla F(\xbf^k)\Vert^2=\frac{1}{2L}\Vert \nabla F(\xbf^k)\Vert^2.
$
 Therefore, we  obtain from \eqref{inequality here to use further} that
\begin{align}
&  \nonumber\frac{1}{2} \E_{\mathcal U^k}\left[\Vert \xbf^{k+1}-\xbf^*\Vert^2\right]- \frac{1}{2}\E_{\mathcal U^k}\left[\Vert \xbf^k-\xbf^*\Vert^2\right]
 \\ \nonumber \leq\,&    -\frac{\gamma_k \mu}{2} \E_{\mathcal U^k}[\Vert \xbf^*-\xbf^k\Vert^2]+\frac{\gamma_kL\delta^2 d  }{2}
+\frac{\eta-1}{2}\E_{\mathcal U^k}\left[\Vert\xbf^{k+1}-\xbf^{k}\Vert^2 \right] \nonumber
\\
&  +\frac{3\gamma_k^2}{2\eta} \left(\frac{{L^2\delta^2} d^2 R}{a_{k+1}\lambda}+\frac{772  R\cdot \ln d}{a_{k+1}\lambda}\cdot  \frac{\sigma^2 }{M}\right) +\gamma_k\lambda\sqrt{s}\cdot \E_{\mathcal U^k} [\Vert \xbf^{k+1}-\xbf^*\Vert] \nonumber
\\&  +\left[\frac{3\gamma_k^2}{\eta} \left(1+\frac{772 R\cdot \ln d}{a_{k+1}\lambda M}\right) L-\gamma_k\right]\cdot \E_{\mathcal U^k}\left[F(\xbf^k)-F(\xbf^*)\right].
\nonumber
\end{align}
By setting $\eta=1/2$, we reduce the above inequality to
\[  
\begin{aligned}
&  \nonumber\frac{1}{2} \E_{\mathcal U^k}\left[\Vert \xbf^{k+1}-\xbf^*\Vert^2\right]- \frac{1}{2}  \Vert \xbf^k-\xbf^*\Vert^2 
\\ \nonumber \leq\,&    -\frac{\gamma_k\mu}{2} \Vert \xbf^*-\xbf^k\Vert^2 +\frac{\gamma_kL\delta^2 d  }{2}
-\frac{1}{4}\E_{\mathcal U^k}\left[\Vert\xbf^{k+1}-\xbf^{k}\Vert^2 \right]
\\
&  + 3\gamma_k^2\left(\frac{{L^2\delta^2} d^2 R}{a_{k+1}\lambda}+\frac{772  R\cdot \ln d}{a_{k+1}\lambda}\cdot  \frac{\sigma^2 }{M}\right) +\gamma_k\lambda \sqrt{s}  \E_{\mathcal U^k}\left[\Vert \xbf^*-\xbf^{k+1}\Vert\right] \nonumber
\\&  +\left(6 \gamma_k^2  L+\frac{4632\gamma_k^2LR}{a_{k+1}\lambda M}-  \gamma_k\right) \cdot  \left[ F(\xbf^k)- F(\xbf^*)\right] \nonumber
 \\ \nonumber \leq\,&    -\frac{\gamma_k\mu}{2} \Vert \xbf^*-\xbf^k\Vert^2 +\frac{\gamma_kL\delta^2 d  }{2}
-\frac{1}{4}\E_{\mathcal U^k}\left[\left(\Vert\xbf^{k+1}-\xbf^{k}\Vert-2\gamma_k\lambda\sqrt{s}\right)^2 \right]
\\
&   -\gamma_k \lambda\sqrt{s}\cdot  \E_{\mathcal U^k}\left[\Vert\xbf^k-\xbf^{k+1}\Vert\right]+\gamma_k^2 \lambda^2s+\gamma_k \lambda\sqrt{s}\cdot \E_{\mathcal U^k}\left[ \Vert\xbf^*-\xbf^{k+1}\Vert\right]\nonumber
\\
&  + 3\gamma_k^2\left[ \frac{{L^2\delta^2} d^2 R}{a_{k+1}\lambda }+\frac{772  R\cdot \ln d}{a_{k+1}\lambda}\cdot  \frac{\sigma^2 }{M} \right] \nonumber
\\&+ \left(6 \gamma_k^2  L+\frac{4632 \gamma_k^2LR}{a_{k+1}\lambda M}- \gamma_k\right)  \cdot \left[  F(\xbf^k)- F(\xbf^*) \right].\nonumber
\end{aligned}
\]
By  strong convexity,  we have $ \Vert\xbf^*-\xbf^{k}\Vert\leq \sqrt{\frac{2}{\mu} \left[F(\xbf^k)-F(\xbf^*)\right]}$. In view of $\E_{\mathcal U^k}\left[ \Vert\xbf^*-\xbf^{k+1}\Vert\right] - \E_{\mathcal U^k}\left[\Vert\xbf^k-\xbf^{k+1}\Vert\right] \leq  \E_{\mathcal U^k}\left[  \Vert\xbf^{k} -\xbf^*\Vert\right]$, thus
\begin{multline}
   \nonumber\frac{1}{2} \E_{\mathcal U^k}\left[\Vert \xbf^{k+1}-\xbf^*\Vert^2\right]- \frac{1}{2} \Vert \xbf^k-\xbf^*\Vert^2 
  \leq\,  -\frac{\gamma_k\mu}{2}\cdot \Vert \xbf^*-\xbf^k\Vert^2 +\frac{\gamma_kL\delta^2 d  }{2}
 \\+\gamma_k \lambda\sqrt{s}\cdot   \left[\sqrt{\frac{2}{\mu} \left[F(\xbf^k)-F(\xbf^*)\right]}\right]+ 3\gamma_k^2\left(\frac{{L^2\delta^2} d^2R}{a_{k+1}\lambda}+\frac{772  R\cdot \ln d}{a_{k+1}\lambda}\cdot  \frac{\sigma^2 }{M}\right) 
 \\+\left(6 \gamma_k^2  L+\frac{4632 \gamma_k^2LR}{a_{k+1}\lambda M}- \gamma_k\right)  \cdot  \left[ F(\xbf^k)- F(\xbf^*)\right]+\gamma_k^2 \lambda^2s.
\end{multline}
Multiplying both sides by $\frac{k+\left\lceil \frac{100L}{\mu\varpi}\right\rceil}{\gamma_k}$ and taking expectation with respect to $\mathcal W:=\left((\xi^{k,m},\ubf^{k,m}):\,k=1,...,K,\,m=1,...,M\right)$, we have
\begin{align}
&  \frac{k+\left\lceil \frac{100L}{\mu\varpi}\right\rceil}{2\gamma_k} \E_{\mathcal W}\left[\Vert \xbf^{k+1}-\xbf^*\Vert^2\right]- \frac{\left(k+\left\lceil \frac{100L}{\mu\varpi}\right\rceil\right) (1-\gamma_k \mu)}{2\gamma_k}\E_{\mathcal W}\left[\Vert \xbf^k-\xbf^*\Vert^2\right]\nonumber
\\ \nonumber &\leq\,    \,  
\left(k+\left\lceil \frac{100L}{\mu\varpi}\right\rceil\right)
\left[\frac{ L\delta^2 d}{2}
+\gamma_k \lambda^2 s+ \gamma_k\left(\frac{{3L^2\delta^2} d^2 R}{a_{k+1}\lambda}+\frac{2316  R\cdot \ln d}{a_{k+1}\lambda}\cdot  \frac{\sigma^2 }{M}\right)\right]
\\\nonumber&    +
\left(k+\left\lceil \frac{100L}{\mu\varpi}\right\rceil\right) \left(6 \gamma_k  L+\frac{4632 \gamma_kLR}{a_{k+1}\lambda M}- 1\right)   \E_{\mathcal W}\left[ F(\xbf^k)- F(\xbf^*)\right]
\\
&   + \left(k+\left\lceil \frac{100L}{\mu\varpi}\right\rceil\right)\cdot \lambda\sqrt{s}\cdot\E_{\mathcal W}\left(  \sqrt{\frac{2}{\mu} \left[F(\xbf^k)-F(\xbf^*)\right]}\right).\label{to invoke recursively}
\end{align}

Note that $\frac{(k+\left\lceil \frac{100L}{\mu\varpi}\right\rceil)(1-\gamma_{k}\mu)}{2\gamma_{k}}=\mu\cdot \frac{(k+\left\lceil \frac{100L}{\mu\varpi}\right\rceil)(k+\left\lceil \frac{100L}{\mu\varpi}\right\rceil-1)}{4} = \frac{k-1+\left\lceil \frac{100L}{\mu\varpi}\right\rceil}{2\gamma_{k-1}}$ 
and $(k+\lceil\frac{100L}{\mu}\rceil)\gamma_k\leq \frac{2}{\mu}$ by the selection of $\gamma_k$ as in \eqref{parameters strongly convex}. Therefore, we may invoke \eqref{to invoke recursively} recursively and sum them up to obtain
\begin{align}
&  \nonumber \frac{\mu\left(K+\left\lceil \frac{100L}{\mu\varpi}\right\rceil\right) }{4} \left(K+\left\lceil \frac{100L}{\mu\varpi}\right\rceil+1\right)\cdot \E_{\mathcal W}\left[\Vert \xbf^{K+1}-\xbf^*\Vert^2\right]\nonumber \\&  -\frac{\mu}{4} \left\lceil \frac{100L}{\mu\varpi}\right\rceil  \cdot \left(\left\lceil \frac{100L}{\mu\varpi}\right\rceil+1\right)\E_{\mathcal W}\left[\Vert \xbf^1-\xbf^*\Vert^2\right]\nonumber
\\&  \nonumber-\sum_{k=1}^K \left(k+\left\lceil \frac{100L}{\mu\varpi}\right\rceil\right)\cdot \left(6 \gamma_kL+\frac{4632 \gamma_k LR}{a_{k+1}\lambda M}-  1\right)   \cdot \E_{\mathcal W}\left[ F(\xbf^k)- F(\xbf^*)\right]
\\ \nonumber \leq\,&    \,\sum_{k=1}^K \left(k+\left\lceil \frac{100L}{\mu\varpi}\right\rceil\right)\left\{\frac{ L\delta^2 d}{2}
+ \lambda\sqrt{s}  \cdot \E_{\mathcal W}\left[\sqrt{\frac{2}{\mu} \left[F(\xbf^k)-F(\xbf^*)\right]}\right]\right\}
\\&  + \sum_{k=1}^K \left[\left(k+\left\lceil \frac{100L}{\mu\varpi}\right\rceil\right)\cdot \left(\frac{{3L^2\delta^2} d^2 R\gamma_k}{a_{k+1}\lambda}+\frac{2316 \gamma_k R\cdot \ln d}{a_{k+1}\lambda}\cdot  \frac{\sigma^2 }{M}\right)\right]+\frac{2K\lambda^2s }{\mu}.\nonumber
\end{align}
Recall that $\lambda= 200\varpi^{-1} K^{-1}L$, and $M=\lceil 50 \varpi K^{3}  \max\{1,\,\sigma^2\}\cdot L^{-3}\mu\cdot \ln d\rceil$, and $K\geq    L^{3/2}R^{1/2}/ \sqrt{\mu}$, then we have $\frac{4632 LR}{\lambda M}\leq \frac{0.47 L^3R}{ K^2\mu}\leq 0.47$. 
Furthermore, $6\gamma_kL\leq 0.12$ by the selection of $\gamma_k$. 
Thus, $6\gamma_k  L+\frac{4632 \gamma_k LR}{a_{k+1}\lambda M}-  1 \leq -0.41$.
Likewise, $\frac{2316 \gamma_k R\sigma^2\cdot \ln d}{a_{k+1}\lambda M}\leq  \frac{0.24L^2R}{K^2\mu}$ and $\frac{3L^2\delta^2Rd^2\gamma_k}{a_{k+1}\lambda}= 0.03\varpi L\delta^2R Kd^2$. 
 Consequently,
\begin{align}
&  \nonumber \frac{\left(K+\left\lceil \frac{100L}{\mu\varpi}\right\rceil\right) \mu}{4} \left(K+\left\lceil \frac{100L}{\mu\varpi}\right\rceil+1\right)\E_{\mathcal W}\left[\Vert \xbf^{K+1}-\xbf^*\Vert^2\right]
\\&  \nonumber-\frac{\mu}{4}  \left\lceil \frac{100L}{\mu\varpi}\right\rceil\left( \left\lceil \frac{100L}{\mu\varpi}\right\rceil+1\right)\E_{\mathcal W}\left[\Vert \xbf^1-\xbf^*\Vert^2\right]
\\&  \nonumber+\sum_{k=1}^K 0.41 \left(k+\left\lceil \frac{100L}{\mu\varpi}\right\rceil\right) \cdot \E_{\mathcal W}\left[ F(\xbf^k)- F(\xbf^*)\right]
\\ \nonumber \leq\,&    \,\sum_{k=1}^K \left(k+\left\lceil \frac{100L}{\mu\varpi}\right\rceil\right) \left\{\frac{ L\delta^2 d}{2}
+ \lambda\sqrt{s} \cdot \E_{\mathcal W}\left[\sqrt{\frac{2}{\mu} \left[F(\xbf^k)-F(\xbf^*)\right]}\right]\right\}\nonumber
\\&  + \sum_{k=1}^K \left(k+\left\lceil \frac{100L}{\mu\varpi}\right\rceil\right)  \left( 0.03\varpi L\delta^2R Kd^2+\frac{0.24L^2R}{K^2\mu}\right) +\frac{4\times10^4\cdot L^2s}{K\mu \varpi^2}.\label{TBC strong convex}
\end{align}
 By the definition of $Y$ and $\mathbf E$, we know that
 \[  
\begin{aligned}
  \mathbf E\left[ F(\xbf^Y)- F(\xbf^*)\right]=&\sum_{k=1}^K\frac{\gamma_{k-1}^{-1}}{\sum_{k=1}^K\gamma_{k-1}^{-1}}\E_{\mathcal W}\left[F(\xbf^k)- F(\xbf^*)\right].\nonumber
\end{aligned}
\]
Similarly,
\[  
\begin{aligned}
  \sqrt{\mathbf E\left[ F(\xbf^Y)- F(\xbf^*)\right]}\geq\,&\mathbf E \sqrt{F(\xbf^Y)- F(\xbf^*)} 
  \\=& \sum_{k=1}^K\frac{\gamma_{k-1}^{-1}}{\sum_{k=1}^K\gamma_{k-1}^{-1}}\E_{\mathcal W}\left[\sqrt{F(\xbf^k)- F(\xbf^*)}\right].\nonumber
\end{aligned}
\]
Recall that  $\gamma_k$ as given in \eqref{parameters strongly convex}, and thus $\sum_{k=1}^K\gamma_{k-1}^{-1}=\frac{K\mu}{2}\left(\left\lceil \frac{100L}{\mu\varpi}\right\rceil+\frac{K+1}{2}\right)$ and $\frac{\gamma_{k-1}^{-1}}{\sum_{k=1}^K\gamma_{k-1}^{-1}}=\frac{k+\left\lceil \frac{100L}{\mu\varpi}\right\rceil}{K\left(\left\lceil \frac{100L}{\mu\varpi}\right\rceil+\frac{K+1}{2}\right)}$. We may then simplify \eqref{TBC strong convex} into 
\[  
\begin{aligned}
    &  \nonumber \frac{\mu\left(K+\left\lceil \frac{100L}{\mu\varpi}\right\rceil\right) }{4} \left(K+\left\lceil \frac{100L}{\mu\varpi}\right\rceil+1\right)\E_{\mathcal W}\left[\Vert \xbf^{K+1}-\xbf^*\Vert^2\right]
\\&    -\frac{\mu}{4}  \left\lceil \frac{100L}{\mu\varpi}\right\rceil \left(\left\lceil \frac{100L}{\mu\varpi}\right\rceil+1\right)\E_{\mathcal W}\left[\Vert \xbf^1-\xbf^*\Vert^2\right]\nonumber
 \\&+ 0.205 K   \left(K+2\left\lceil \frac{100L}{\mu\varpi}\right\rceil+1\right)  \mathbf E\left[ F(\xbf^Y)- F(\xbf^*)\right]
\\ \nonumber \leq\,&    \,  \frac{K}{2}\left(K+1+2\left\lceil \frac{100L}{\mu\varpi}\right\rceil\right) \left\{\frac{L\delta^2 d}{2}
+  \lambda\sqrt{s} \sqrt{\frac{2}{\mu} \mathbf E\left[F(\xbf^Y)-F(\xbf^*)\right]}\right\}\nonumber
\\&    +  \frac{K}{2} \left(K+1+2\left\lceil \frac{100L}{\mu\varpi}\right\rceil\right)\cdot\left( 0.03 \varpi L\delta^2R Kd^2+\frac{0.24 L^2R}{K^2\mu} 
\right) +\frac{4\times10^4\cdot L^2s}{K\mu \varpi^2}.
\end{aligned}
\]
By rearranging the items and plugging in $\lambda=\frac{200 L}{K}$ from \eqref{parameters strongly convex}, we have
\begin{align}
&  \nonumber  0.205K \left(K+2\left\lceil \frac{100L}{\mu\varpi}\right\rceil+1\right)  \mathbf E\left[ F(\xbf^Y)- F(\xbf^*)\right]
\\ \nonumber \leq\,&    \,\frac{\mu}{4}  \left\lceil \frac{100L}{\mu\varpi}\right\rceil  \left(\left\lceil \frac{100L}{\mu\varpi}\right\rceil+1\right)\E_{\mathcal W}\left[\Vert \xbf^1-\xbf^*\Vert^2\right]
\\&+\frac{K}{2} \left(K+1+2\left\lceil \frac{100L}{\mu\varpi}\right\rceil\right) \frac{L\delta^2 d}{2}\nonumber
\\&    \nonumber
+100K\left(K+2\left\lceil \frac{100L}{\mu\varpi}\right\rceil+1\right)\cdot  \frac{L}{K}\cdot \sqrt{\frac{s}{\mu}} \sqrt{2\mathbf E\left[F(\xbf^Y)-F(\xbf^*)\right]}
\\&     + \frac{1}{2}K \left(K+1+2\left\lceil \frac{100L}{\mu\varpi}\right\rceil\right)\cdot\left( 0.03\varpi L\delta^2R Kd^2+\frac{ 0.24L^2R}{K^2\mu}  \right)   +\frac{4\times10^4\cdot L^2s}{K\mu\varpi^2}.\nonumber
\end{align}
Dividing both sides by $\frac{K}{2} \left(K+2\left\lceil \frac{100L}{\mu\varpi}\right\rceil+1\right)$, we then have
\begin{align}
&  \nonumber   0.41 \mathbf E\left[ F(\xbf^Y)- F(\xbf^*)\right]
\\ \nonumber \leq\,&    \,
\frac{200L}{K}  \sqrt{\frac{s}{\mu}} \cdot \sqrt{2\mathbf E\left[F(\xbf^Y)-F(\xbf^*)\right]}+\frac{L \delta^2 d }{2}
\\\nonumber
&  + \frac{1}{K\left(K+2\left\lceil \frac{100L}{\mu\varpi}\right\rceil+1\right)}\cdot\frac{\mu}{2}  \left\lceil \frac{100L}{\mu\varpi}\right\rceil \left(\left\lceil \frac{100L}{\mu\varpi}\right\rceil+1\right)\E_{\mathcal W}\left[\Vert \xbf^1-\xbf^*\Vert^2\right]
\\&    +   0.03\varpi L\delta^2R Kd^2+ \frac{0.24L^2R}{K^2\mu}
      +\frac{8\times 10^4\cdot L^2s}{\mu K^2\varpi^2\left(K+2\left\lceil \frac{100L}{\mu\varpi}\right\rceil+1\right)}.\nonumber
\end{align}
We can view  the above equation as  a quadratic inequality with the unknown variable $\sqrt{\mathbf E[F(\xbf^Y)-F(\xbf^*)]}$. Solving this inequality and after some  simplification, we obtain, for some constant $c_3>0$, $
 \nonumber  \mathbf E[ F(\xbf^Y)- F(\xbf^*)]
 \leq\,    c_3\cdot\left[
\frac{L^2(s+R)}{K^2\mu} +\frac{L^2}{\mu K^2}\cdot\E_{\mathcal W}[\Vert \xbf^1-\xbf^*\Vert^2] 
+ L \delta^2 d+ L\delta^2d^2 RK\varpi+ \frac{L^2s}{\mu \varpi^2 K^3}\right]$. Recall that $R\geq 1$, $\delta \leq \frac{\theta}{ K^{1.5}  d  }$ and $\xbf^1$ is deterministic, then  it directly leads to the desired result.  \qed
\end{proof}
 \begin{remark}
 Below, we would like to make a few remarks on the above theorem. For these remarks, we assume w.l.o.g. that $\mu\leq 1$ and $\sigma\geq 1$.
 \begin{itemize}
 \item 
 Let $D_0:=\Vert \xbf^1-\xbf^*\Vert^2$, which can be $\mathcal O(s)$ under Assumption \ref{sparsity assumption}, if $\Vert \xbf^1\Vert_0\leq s$. (For instance, $\xbf^1$ can be an all-zero vector.) 
 Compared to \eqref{oracle complexity detailed}, For any $\epsilon\in(0,\,L^{-1}]$, Theorem \ref{main theorem strongly convex} implies a significantly sharper query complexity
$
\mathcal O\left(\frac{L  \sigma^2(s+R+D_0)^2}{\mu\epsilon^2}\ln d\right).
$
To see this, note that the iteration complexity is $\mathcal O\left(\sqrt{\frac{L^2}{\mu\epsilon}\cdot(s+R+D_0)}\right)$, and the per-iteration query complexity is $M=\mathcal O\left(\frac{ (s+R+D_0)^{3/2}\sigma^2}{\epsilon^{3/2}\mu^{1/2}}\ln d\right)$. Once again, the query complexity is independent of dimensions $d$, up to a logarithmic term, when $\sigma$, $s$, $L$, and $\mu$ are fixed.
\item By a similar  argument as used in deriving \eqref{markov bound}, there exists some universal constant $C_3>0$, such that
\begin{multline}
    \text{Prob}\left[ F(\xbf^Y)- F(\xbf^*)\leq \varepsilon\right]
    \\\geq      1-C_2\cdot \frac{L^2}{K^2\mu}\left[s+R \cdot (1+\theta^2\varpi\cdot \mu)+ \Vert \xbf^1-\xbf^*\Vert^2  + \frac{\mu \theta^2}{KLd}+\frac{1}{\varpi^2 K}\right],\label{markov bound strongly convex}
\end{multline}
for any $\varepsilon>0$, under the same set of assumptions as in Theorem \ref{main theorem strongly convex}.
\item Similar to the convex S-ZOO case, to implement Algorithm \ref{main-alg} for solving the strongly convex S-ZOO, we do not need to know the sparsity-level $s$ of the optimal solution. Actually,  the algorithm can automatically exploit the sparsity, provided  a coarse over-estimate of $R$,  i.e., the $\ell_1$-norm of the optimal solution, is available.

\item Also similar to the convex S-ZOO case,  the effectiveness of the proposed algorithm depends on machine precision,   $\hat\epsilon$.  For the double precision on a 32-bit computer with   $\hat\epsilon=2^{-52}\approx 10^{-16}$, it is stipulated that $Kd^{1.5}< \frac{\theta}{\hat\epsilon}\approx 10^{16}\cdot \theta$; namely, there can be an upper limit on the admissible problem dimensionality for the proposed SI-SGF.

\item The  algorithm parameters can be more flexible than \eqref{parameters strongly convex}  to achieve the promised query complexity.  In fact, if we choose any of $\lambda$, $a$, $\gamma$, $\delta$, or $M$ to be some constant multiple of their current values, the same complexity rate holds.
\end{itemize}
\end{remark} 
 \subsection{Alternative schemes for algorithm output}\label{sec alternative output}
 Algorithm \ref{main-alg} relies on a simple and randomized criterion to determine the output 
 $\xbf^Y$ from  the sequence $\{\xbf^k\}$, with the index $Y$ randomly chosen as per a pre-defined discrete distribution. We may also use two alternative output schemes (AOS). The first AOS generates $\xbf^{k^*}$ as follows.
     \begin{align} k^*\in \arg\min\left\{M^{-1}\sum_{m=1}^Mf(\xbf^k,\,\xi^{k,m}):\,k=1,...,K\right\}.\label{AOS-2}\end{align}
When the above set is not a singleton,  $k^*$ is selected arbitrarily from the set.

Intuitively, this AOS outputs the solution with the smallest in-sample cost calculated on a mini-batch among all the solutions generated from iterations 1 to $K$.  Our numerical experiments in Section \ref{sec: numerical} show that the AOS tends to yield better solution quality in practice than the default randomized output scheme in  Algorithm \ref{main-alg}. The corollary below
provides a theoretical guarantee of the AOS's effectiveness.

\begin{corollary}\label{first corollary}
Let $\theta$ and $\varpi$ in \eqref{parameter settings} and \eqref{parameters strongly convex} be some universal constants. For any  $\varepsilon>0$, there exists some constant  $C_4>0$ such that 
the following  hold.
\begin{enumerate}
 \item[(a)] Under the same setting as in Theorem \ref{main theorem},  it holds with probability at least $1-\frac{C_3}{\varepsilon}\cdot \left(\frac{L^2}{K}+  \frac{L  \Vert \xbf^1-\xbf^*\Vert^2 }{ K} + \frac{1+ L/K+LR}{ K}\right)$ that 
$F(\xbf^{k^*})-F(\xbf^*)\leq 3\varepsilon$.
 \item[(b)] Under the same setting as in Theorem \ref{main theorem strongly convex},  $F(\xbf^{k^*})-F(\xbf^*)\leq 3\varepsilon$
 with probability at least $1-\frac{C_4}{\varepsilon} \cdot  \frac{L^2}{K^2\mu}\left(L+s+R + \Vert \xbf^1-\xbf^*\Vert^2  \right).$
 \end{enumerate}
\end{corollary}
\begin{proof} Because
$M^{-1}\sum_{m=1}^Mf(\xbf^{k^*},\xi^{k^*,m})
 \leq M^{-1}\sum_{m=1}^Mf(\xbf^{Y},\xi^{Y,m})$, we have
 \begin{align}
 & F(\xbf^{k^*})-F(\xbf^{Y})\nonumber
 \\  
 \leq~&    F(\xbf^{k^*})-F(\xbf^{Y})+M^{-1}\sum_{m=1}^Mf(\xbf^{Y},\xi^{Y,m})-M^{-1}\sum_{m=1}^Mf(\xbf^{k^*},\xi^{k^*,m})\nonumber
 \\  
  \leq~&    \sqrt{\left(F(\xbf^{k^*})- \sum_{m=1}^M\frac{f(\xbf^{k^*},\xi^{k^*,m})}{M}\right)^2}+\sqrt{\left(F(\xbf^{Y})- \sum_{m=1}^M\frac{f(\xbf^{Y},\xi^{Y,m})}{M}\right)^2}\nonumber
  \\   \leq~&    2\max_{k=1,...,K}\sqrt{\left(F(\xbf^{k})-M^{-1}\sum_{m=1}^Mf(\xbf^{k},\xi^{k,m})\right)^2}.\label{to use corollary 1}
 \end{align}
Under Assumption \ref{mean and variance},
 $\E_{(\xi^{k,m}:\,m=1,...,K)}[\sqrt{(F(\xbf^{k})-M^{-1}\sum_{m=1}^Mf(\xbf^{k},\xi^{k,m}))^2}]\leq \frac{\sigma^2}{M}$, for  $k=1,...,K$.
 By Markov's inequality, for any $\varepsilon>0$, it holds with probability at least $1-\frac{\sigma^2}{M \varepsilon}$ that $\sqrt{\left(F(\xbf^{k})-M^{-1}\sum_{m=1}^Mf(\xbf^{k},\xi^{k,m})\right)^2}\leq \varepsilon$.  Furthermore, by union bound and De Morgan's law, we then have
$ \max_{k=1,...,K} \sqrt{\left(F(\xbf^{k})-M^{-1}\sum_{m=1}^Mf(\xbf^{k},\xi^{k,m})\right)^2}\leq \varepsilon$
 with probability at least $1-\frac{K\sigma^2}{M\varepsilon}$. This combined with \eqref{to use corollary 1} implies that
$F(\xbf^{k^*})-F(\xbf^{Y}) \leq  2\varepsilon$
 with probability at least $1- \frac{K\sigma^2}{M\varepsilon}$. In view of \eqref{markov bound} and \eqref{markov bound strongly convex} as well as the choices of $M$ in  \eqref{parameter settings} and \eqref{parameters strongly convex}, we then have the desired results in (a) and (b), respectively.\qed
\end{proof}
  From this corollary, we know that the AOS is provably effective. It is also worth noticing that adopting this output  scheme   would incur almost no additional computational cost in Algorithm \ref{main-alg}. Our numerical results presented subsequently indicate the effectiveness compared with the  output scheme originally in Algorithm \ref{main-alg}.

 The second AOS is commonly utilized in the literature \cite{nemirovski2009robust}. At the end of iteration $K$, Algorithm \ref{main-alg} yields $\bar\xbf$, which is the weighted average of  the whole solution sequence calculated as
 \begin{align}
 \bar\xbf :=\sum_{k=1}^K\frac{\gamma_{k-1}^{-1}\cdot \xbf^k}{\sum_{k=1}^K\gamma_{k-1}^{-1}}.\label{second AOS}
 \end{align}
To understand the effectiveness of this AOS, we observe that  $\mathbb E_Y[\xbf^Y]= \bar\xbf$, where $Y$ and $\xbf^Y$ are defined as in   Theorems \ref{main theorem} and \ref{main theorem strongly convex} and $\mathbb E_Y$ denotes the expectation taken over $Y$.  Since $\mathbb F$ is convex, the above  leads to $\mathbb E_Y\left[\mathbb F(\xbf^Y)\right]\geq \mathbb F\left(\mathbb E_Y[\xbf^Y]\right)=\mathbb F(\bar\xbf)$.  This, combined with Theorems \ref{main theorem} and \ref{main theorem strongly convex},  immediately shows the effectiveness of the AOS in \eqref{second AOS}.

\section{Numerical Results}\label{sec: numerical} 
 We conducted   preliminary experiments on a stochastic quadratic programming problem modified from \cite{nesterov2017random}. More specifically, we focused on solving the following optimization problem:
\begin{multline}
    f_d\left(\xbf;(\omega_i,\upsilon_i)\right)=\frac{1}{2} x_1^2+\sum_{i=1}^{d-1}\frac{1}{2}(x_{i+1}-x_i-C_{i+1}+C_i)^2
    \\+\frac{1}{2}x^2_d +\sum_{i=1}^d \omega_i
   \cdot \upsilon_i\cdot x_i,\label{population level LP}
\end{multline}
where   $x_i$ is the $i$-th entry of $\xbf$, $C_i=1.5$ for all $i\in\{2, 6, 9\}$, and $C_i=0$ for all other $i$ (that is, $i\notin\{2,\,6,\,9\}$), for each $i$, $\omega_i$ is a standard normal random variable, and $\vbf:=\{\upsilon_i\}$ is a vector of   random variables such that exactly three components of it take value 1. That means, each $\vbf$ is randomly drawn from  $\mathcal{V}:=\{\upsilon_i\in \{0,1\}: \sum_{i=1}^{d}{\upsilon_i}=3\}$ with equal probability of $\frac{1}{|\mathcal{V}|}$. By  construction, the optimal solution of the problem is verifiably $\xbf^*=[0;\,1.5;\,0;\,0;\,0;\,1.5;\,0;\,0;\,1.5;\,0;\,...;\,0]$, which is indeed a sparse solution,  the corresponding optimal objective value is $0$, and $\sigma^2=3$.  The suboptimality gap, in this case, is the objective value of the output solution.

 In all our experiments, we set the budget of the maximum number of zeroth-order oracle calls to be  320,000.   We tested different mini-batch sizes $M$ for SI-SGF under both convex and strongly convex settings in different problem cases. Correspondingly,  the maximum iteration count was
 \begin{align}K=\left\lfloor\frac{320,000}{M}\right\rfloor.\label{K-M relationship}
 \end{align}



We experimented with the following algorithms and configurations:
\begin{itemize}
\item SI-SGF for the convex settings with $\varpi=5$, $(\lambda,\gamma,a)$ as per \eqref{parameter settings} in Theorem \ref{main theorem}, and $M$ determined empirically in the sequel:
\begin{description}
    \item[~\it si-sgf$^{R}$:]   SI-SGF with randomized output scheme as in Algorithm \ref{main-alg};
    
    \item[~\it si-sgf$^{*}$:]  SI-SGF with output scheme as in \eqref{AOS-2};
    \item[~\it si-sgf$^{A}$:]  SI-SGF with output scheme as in \eqref{second AOS};
     \end{description}
     \item SI-SGF for the strongly convex settings with   $\varpi=5$, $(\lambda,\gamma_k,\,a_k)$   selected as per \eqref{parameters strongly convex} in  Theorem \ref{main theorem strongly convex}, and $M$ determined empirically below:
     \begin{description}
     \item[~\it si-sgfs$^{R}$:]   SI-SGF with randomized output scheme as in Algorithm \ref{main-alg};
   
    \item[~\it si-sgfs$^{*}$:]  SI-SGF with output scheme as in \eqref{AOS-2};
     \item[~\it si-sgfs$^{A}$:]  SI-SGF with  output scheme as in \eqref{second AOS};
    \end{description}

    \item Benchmark algorthm:
    
    \begin{description}
\item[~\it sgf:]  The SGF from \cite{ghadimi2013stochastic} with the best combination of the output scheme and the mini-batch size. More precisely, for each problem instance, all combinations of the three aforementioned output schemes and the candidate mini-batch sizes  (to be detailed below) were compared. Among them, the  one with the  best quality  in terms of the expected cost function $F$  was selected as the {\it sgf}'s output solution.  The step sizes for SGF with $M=1$ were chosen as per Corollary 3.3 of \cite{ghadimi2013stochastic} with $\bar D=1.5$ therein. Denote this value by  $\gamma_{SGF}$. Then, for SGF with other mini-batch sizes $M$, the step sizes were selected to be $\gamma_{SGF}\cdot M$.

%

\end{description}
\end{itemize}
All the algorithms above were initialized with an all-zero vector, whose objective value was $6.75$.

Note that  the alternative dimension-insensitive S-ZOO algorithms by \cite{cai2020zeroth} and \cite{balasubramanian2018zeroth_neurips} cannot be applied to our settings directly, because the problem considered here does not satisfy the additive structure or the assumption of everywhere sparse  gradient.

The first experiment was to determine the mini-batch size $M$ for SI-SGF and understand how  the performance of the algorithms above changes as the  mini-batch size $M$ varies. To this end, we fixed $d$ and $\delta$ to be $2^{15}$ and $10^{-7}$, respectively. This  dimensionality was intentionally chosen to be larger than a tenth of the total budget of calls to the zeroth order oracle. For each choice of $M$, we performed ten random replications.  Mean values and standard deviations of the suboptimality gaps are shown in both Figure \ref{fig: subopt vs M}.(a) and Table \ref{tab: increase M} (in the supplemental material). As shown therein,   `{\it si-sgf$^*$}', `{\it si-sgfs$^*$}', and `{\it si-sgfs$^A$}' were relatively insensitive to different choices of   $M$, especially when $M\geq 100$. Some deterioration in the performance of these three variants has been observed for scenarios with larger values of $M$. Recall that $K$, the maximum iteration number, was decreasing in $M$ as per \eqref{K-M relationship} to maintain the same maximum number of queries. Thus, the above observed deterioration was believed to be  the  result of smaller values of $K$. Other variants of the SI-SGF were comparatively more sensitive to the changes in $M$.  Yet, in almost all test cases, all variants of SI-SGF outperformed the benchmark `{\it sgf}'.   We further evaluated the average   suboptimality gaps across all three output schemes for the SI-SGF, and picked the mini-batch sizes that led to the best performance.   As in Figure \ref{fig: subopt vs M}.(b) and Table \ref{tab: increase M}, the  best mini-batch sizes in this test  were 160 and 280, respectively, for SI-SGF under convex and strongly convex settings. 
 
\begin{figure}
\begin{tabular}{cc}
\includegraphics[width=0.51\textwidth]{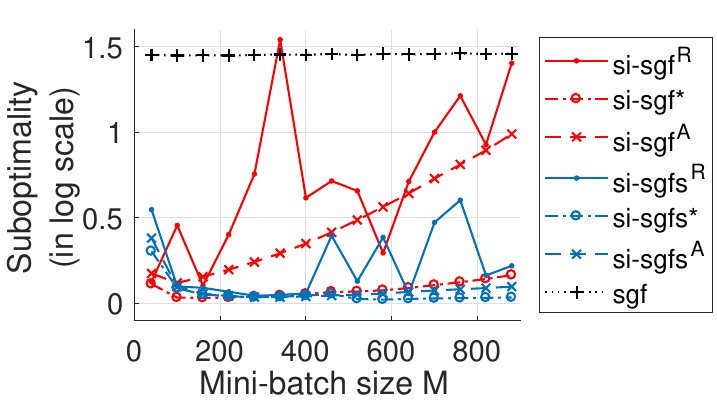} &
\includegraphics[width=0.46\textwidth]{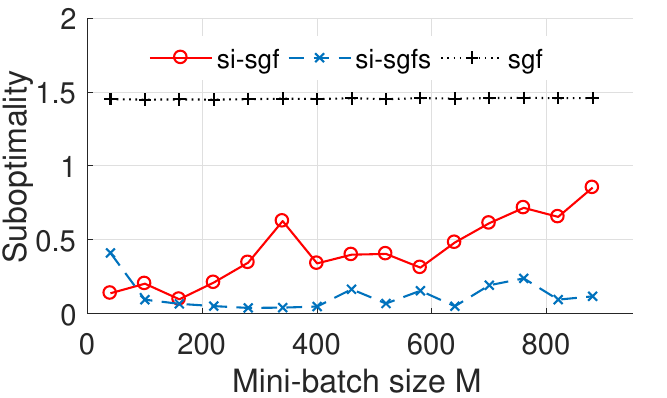}
\\(a)&(b)
\end{tabular}
\caption{\label{fig: subopt vs M} Comparisons in the average  suboptimality gaps (over ten random replications) of all the algorithms  for different mini-batch sizes $M$. 
Subplot (a)  shows the mean suboptimality gaps of each of the algorithms. Subplot  (b)  shows the comparisons among  `{\it si-sgf}', `{\it si-sgfs}', and `{\it sgf}'. Here,  `{\it si-sgf} ' refers to the average of the mean suboptimality gaps generated by `{\it si-sgf$^R$}', `{\it si-sgf$^*$}', and `{\it si-sgf$^A$}'; `{\it si-sgfs}' refer to the average of the mean suboptimality gaps generated by `{\it si-sgfs$^R$}', `{\it si-sgfs$^*$}', and `{\it si-sgfs$^A$}'.}
\end{figure}

 The second set of experiments was to test the  algorithms when the dimensionality $d$ belonged to $\{
2^k: 6\leq k \leq 21\}$. We set $\delta=10^{-7}$, and the mini-batch sizes were 160 and 280 (as selected above) for the SI-SGF in convex and strongly convex settings, respectively. For each case, the benchmark `{\it sgf} ' reported in this test was the best suboptimality gap achieved by the SGF's output among all combinations of the three different output schemes and the mini-batch sizes of 1, 160, and 280. For each case, five random replications were performed.  Mean values and standard deviations of the resulting suboptimality gaps out of these replications  are reported in both Table \ref{tab: increase d} (in the supplemental material of this paper) and  in Figure \ref{fig: subopt vs dim}.  It can be seen from subplots (a)-(c) of this figure, as the dimensionality $d$ increased exponentially above $10^{14}$, the performance of the benchmark `{\it sgf} ' deteriorated rapidly and then plateaued as the suboptimality gap got closer to 6.75, which is the objective value of the initial solution. In contrast, the proposed SI-SGF under both convex and strongly convex settings for all three output schemes was significantly insensitive to the increase of dimensions. These observations agreed with our theoretical results that the SI-SGF, under both convex and strongly convex settings, are provably dimension-insensitive. 

 Figure \ref{fig: subopt vs dim}.(d). compares different variants of SI-SGF. Each curve therein shows the ``ratio of gaps'', that is, the ratio of the suboptimality gaps incurred by an SI-SGF variant of interest to that of `{\it si-sgfs}$^*$' when $d$ increased.  If any point is above  the  line of $y=1$, then that SI-SGF variant performed worse than `{\it si-sgfs}$^*$' in the corresponding case of $d$. As can be seen from Subplot (d), `{\it si-sgfs}$^*$' yielded the best overall performance among all the variants. Both `{\it si-sgf}$^*$' and `{\it si-sgfs}$^A$'  were competitive against `{\it si-sgfs}$^*$', yet `{\it si-sgfs}$^*$' was noticeably better when $d\geq 2^{19}$. The rest of the variants were non-trivially less competitive in almost all the cases of $d$.

\begin{figure}
\begin{tabular}{cc}
\includegraphics[width=0.48\textwidth]{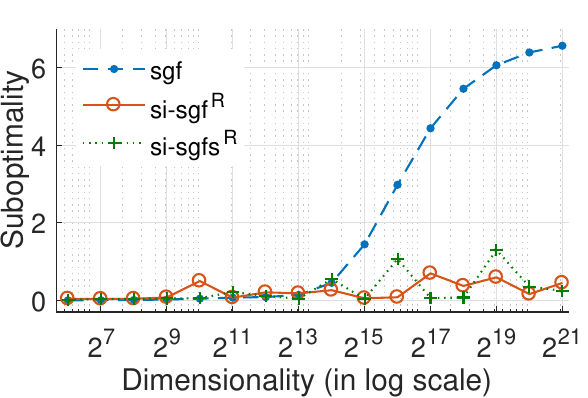} &
\includegraphics[width=0.48\textwidth]{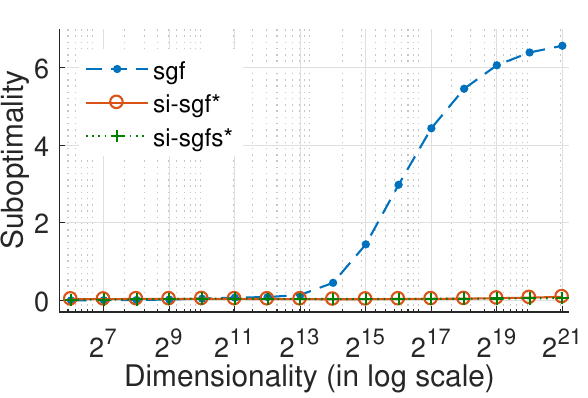}
\\(a)&(b)
\\
\includegraphics[width=0.48\textwidth]{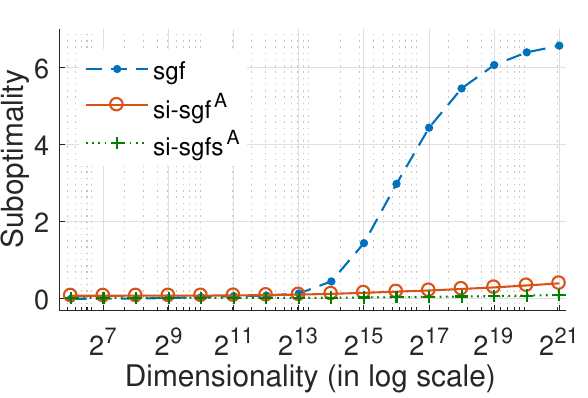}&
\includegraphics[width=0.48\textwidth]{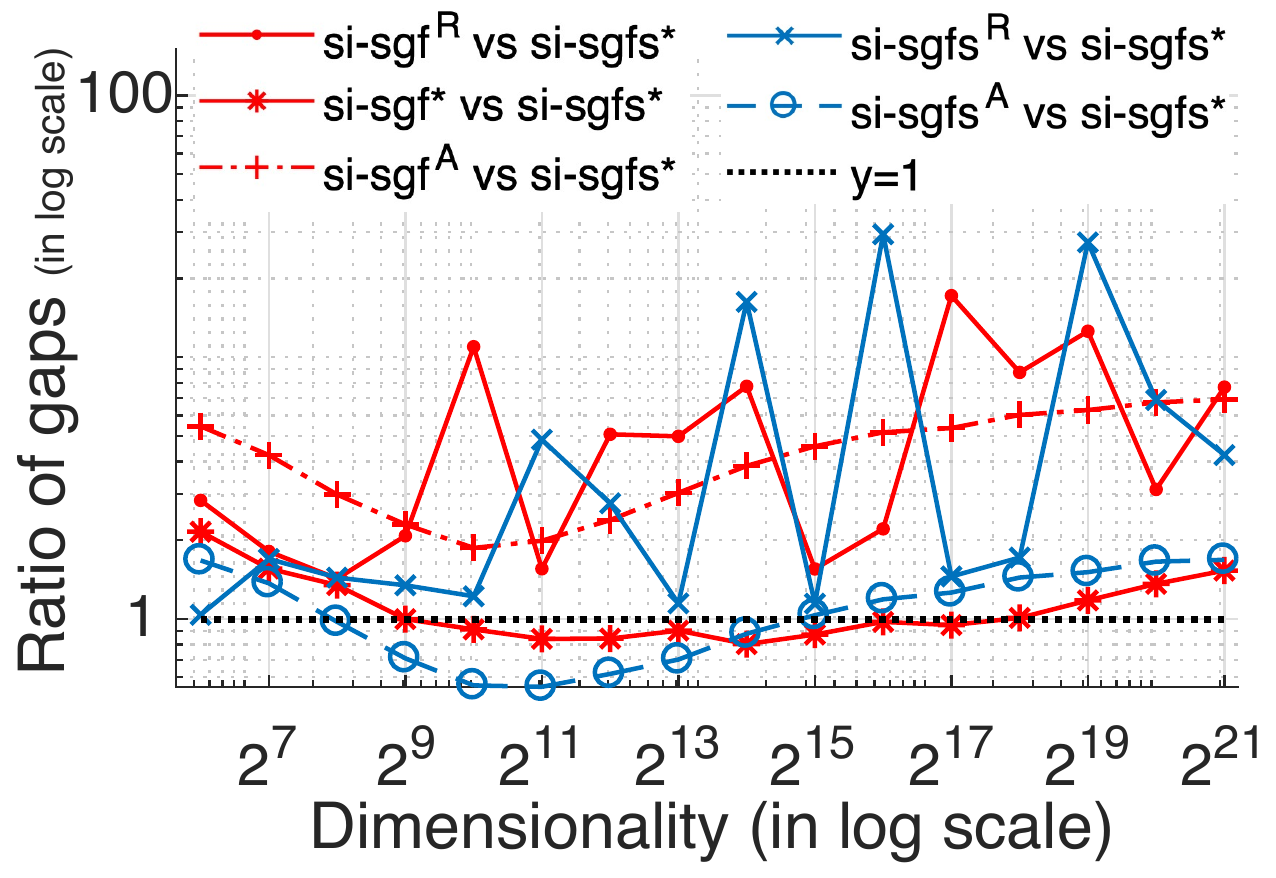}
\\(c)&(d)

\end{tabular}
\caption{\label{fig: subopt vs dim} (a). The mean suboptimality gaps of `{\it si-sgf}$^R$' and `{\it si-sgfs}$^R$', i.e., SI-SGF in both convex and strongly  settings with randomized output scheme, in comparison with  `{\it sgf} '. (b).  The mean suboptimality gaps of `{\it si-sgf}$^*$' and `{\it si-sgfs}$^*$', i.e., SI-SGF in both convex and strongly  settings with output scheme as in \eqref{AOS-2}, in comparison with  `{\it sgf} '. (c).  The mean suboptimality gaps of `{\it si-sgf}$^A$' and `{\it si-sgfs}$^A$', i.e., SI-SGF in both convex and strongly  settings with output scheme as in \eqref{second AOS}, in comparison with `{\it sgf} '. (d). The ratios between the mean suboptimality gaps of different variants of SI-SGF and  that of `{\it si-sgfs}$^*$'. For subplots (a)-(c), $6.75$ was the objective value of the initial solution to all algorithms.}
\end{figure}

The last experiment was focused on the sensitivity of the algorithms to the hyper-parameter $\delta$, which is used in   the randomized smoothing  scheme \eqref{randomized smoothing} for gradient estimation.  We set $d=2^{15}$ and tested the scenarios with the value of $\delta$  ranging from   $10^{-7}$ to $10^{-3}$. The mini-batch sizes  were 160 and 280, respectively, for SI-SGF under convex and strongly convex settings according to the first experiment above. Figure \ref{fig: subopt vs delta} and Table \ref{tab: increase delta} (in the supplemental material of this paper) summarize the results. Our benchmark `{\it sgf}', again, denotes the best suboptimality gap achieved by the SGF's output among all combinations of the three different output schemes and the mini-batch sizes of 1, 160, and 280. Each entry in this table reports the mean and standard deviation of suboptimality gaps generated by different algorithms with different  $\delta$ over 10 random replications. As one can see, the results were comparable when $\delta\leq 10^{-6}$ for all the algorithms. However,    significant performance deterioration was observed, when $\delta$ became larger. The canonical SGF appeared to be the most insensitive towards $\delta$; the deterioration  did not happen until $\delta$ was $10^{-3}$. In contrast, for  `{\it si-sgfs$^{R}$}', the observed deterioration started as $\delta$ became no less than $10^{-5}$. For all other variants of SI-SGF, the deterioration started when $\delta$ turned no less than $10^{-4}$.  Nonetheless, in all cases, the variants of SI-SGF significantly outperformed the SGF in  terms of the suboptimality. 

\begin{figure}
 \centering
\includegraphics[width=0.5\textwidth]{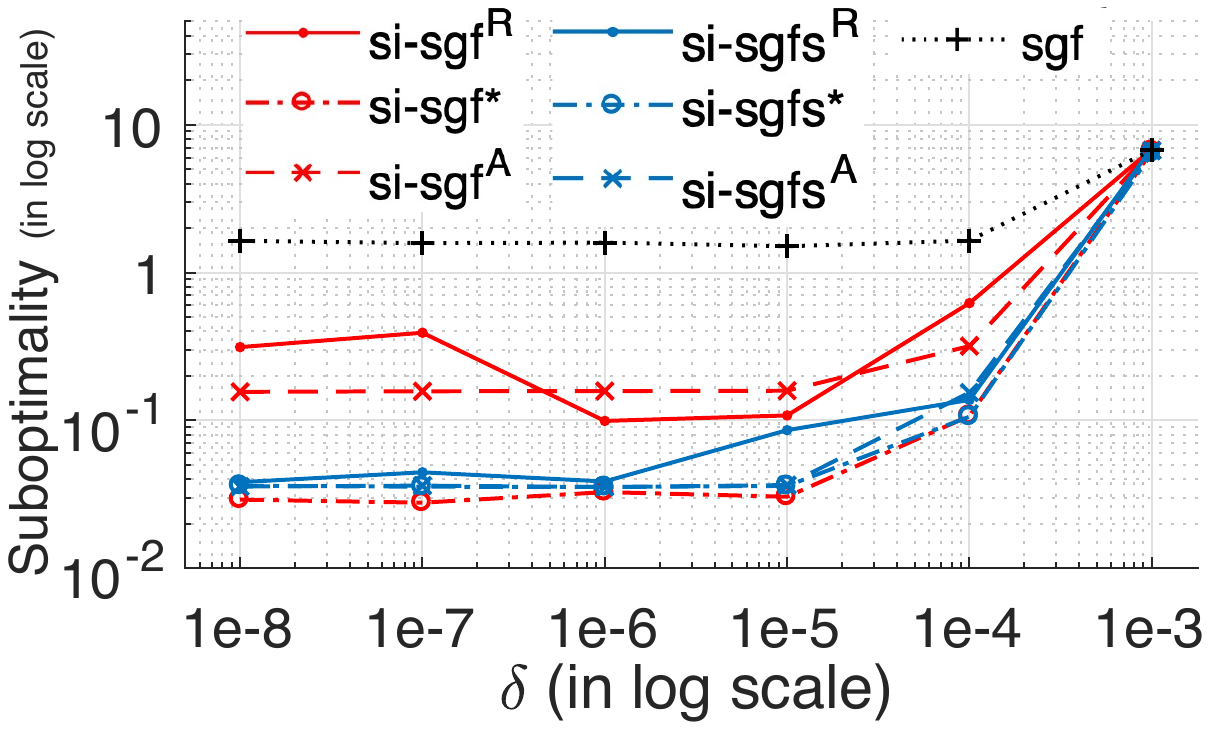} 
\caption{\label{fig: subopt vs delta} Comparisons among all the algorithms in the average suboptimality gaps out of ten random replications  for different values of the hyper-parameter $\delta$.}
\end{figure}

\section{Concluding remarks}\label{sec: conclusions}
This paper presents a sparsity-inducing stochastic gradient-free (SI-SGF) algorithm for solving high-dimensional S-ZOO problems. By exploiting  (weak) sparsity, the proposed algorithm is significantly less sensitive to the increase of dimensionality. In contrast to all existing dimension-insensitive S-ZOO paradigms, our theories do not require the (sometimes critical) assumptions such as everywhere sparse or compressible gradient. Our numerical results indicate that the proposed SI-SGF is a promising approach and can potentially outperform the baseline stochastic gradient-free algorithms.

\begin{appendix}
\section{Technical proofs}\label{sec: proof}
 
 \subsection{Proof of Proposition \ref{lemma1}}\label{Proof of proposition 1}

\begin{proof}
By construction in Step 3 of Algorithm \ref{sub-alg}, it must hold that either $\widetilde v_i\geq U$ or $\widetilde v_i=0$ for all $i=1,..., 2d$. This immediately implies that either $\vert v_i\vert\geq U$ or $v_i=0$ for $i=1,...,d$, which shows Part (a) of the proposition.

The proof for Part (b) is divided into two steps below.

{\bf Step 1.} In this step, we would like to first show that 
$\widetilde{\vbf}$, as in Algorithm  \ref{sub-alg}, is a KKT point of 
\begin{align}\label{lemma1-KKT-solution}
\min_{\zbf\in \R_+^{2d}}\,\left\{\frac{1}{2\gamma}\Vert \mathbf z-\widetilde{\xbf} \Vert^2+\sum_{i=1}^{2d} P_\lambda(z_{i}):\,\mathbf 1^\top\mathbf z\leq R\right\},
\end{align}
where $P_\lambda(\theta):=\int_{0}^\theta\frac{[a\lambda-t]_+}{a}\,\text{d}t$ for arbitrary values of $a$, $\lambda$, and $\gamma$ such that $U=a\lambda$ and $a\leq \frac{\gamma}{2}$.    More explicitly, we will show that  there exist some $\beta$ and $(\mu_i)$ such that $\widetilde \vbf=(\widetilde v_i)$ satisfies the following nonlinear system:
\begin{align}
\begin{split}
  \frac{1}{\gamma}( \widetilde v_{i}-\widetilde x_{i})
+\frac{[a\lambda- \widetilde v_{i}]_+}{a}
+\beta-\mu_{i}=\,&  \,\,0,\qquad  i=1,...,2d; \\
   \widetilde v_{i}\geq 0,\,\,\mu_{i}\geq \,&\,\,0,\qquad i=1,...,2d;
   \\\,\, \mu_{i}\cdot\widetilde  v_{i}=\,&  \,\,0,\qquad i=1,...,2d; 
\\
  \beta\geq0,~~\beta\cdot\left(\sum_{i=1}^{2d}  \widetilde v_{i}-  R\right)=\,&  \,  \,0,~~\sum_{i=1}^{2d}  \widetilde v_{i}\leq  R.
\end{split}\label{KKT-conditions}
\end{align}


Let $\zbf$ be the result computed in Step 2 of Algorithm \ref{sub-alg}. We consider two cases below.

 \noindent {\it Case (i): Consider the case where $\1^\top\zbf \leq R$:}

\noindent According to Algorithm \ref{sub-alg},   $\widetilde v_i=z_{i}$  for $i=1,...,2d$.   Then we can set $\beta=0$, and  let
$\mu_i=\begin{cases}\lambda-\frac{\widetilde x_i}{\gamma}, &\text{if $\widetilde x_{i}< U$},
\\
0, &\text{otherwise}, 
\end{cases}
$    for all $i:\,1\leq i \leq 2d$. By Steps 2 and 3 of Algorithm \ref{sub-alg}, we evidently have $\mu_i\cdot \widetilde v_i=0$ for all $i$ and $\1^\top\widetilde \vbf\leq R$.  Thus, the above construction immediately leads to the third and the last lines of \eqref{KKT-conditions}. Below, we will prove the first two lines of \eqref{KKT-conditions}.

For all $i$ such that $\widetilde x_{i} \geq U$, by the construction both in the above and in Steps 2 and 3 of Algorithm \ref{sub-alg}, we have $\widetilde v_{i}=  z_{i} = \widetilde x_{i}\geq U$, $\beta= 0$ and  $\mu_{i}=0$.  Thus, the second line of \eqref{KKT-conditions} holds for all  $i$ such that $\widetilde x_{i} \geq U$. Similarly, 
$
 \frac{1}{\gamma}( \widetilde v_{i}-\widetilde x_{i})+\frac{[a\lambda- \widetilde v_{i}]_+}{a}+\beta-\mu_{i}= \frac{1}{\gamma}( \widetilde v_{i}-\widetilde x_{i})+\beta-\mu_{i} = 0$, which shows that the first line of \eqref{KKT-conditions} holds for $i$ such that $\widetilde x_{i} \geq U$.    
 
 For all $i$ such that $\,\widetilde x_{i} < U$, we have $\widetilde v_{i}=z_{i} = 0$, $\beta = 0$, and $\mu_{i}= \lambda-\frac{\widetilde x_{i}}{\gamma}$. As a result, 
$
 \frac{1}{\gamma}( \widetilde v_{i}-\widetilde x_{i})+\frac{[a\lambda- \widetilde v_{i}]_+}{a}+\beta-\mu_{i}=-\frac{1}{\gamma}\widetilde x_{i}+\lambda+\beta-\mu_{i}= 0$, which shows that the first line of \eqref{KKT-conditions} holds for $i$ such that $\widetilde x_{i} < U$.  In view of $ \gamma \geq 2a $ and $U=a\lambda$, we have  $\gamma\lambda-\widetilde x_{i}\geq U-\widetilde x_{i} > 0$. Thus, $\mu_{i}=\lambda-\frac{\widetilde x_{i}}{\gamma}\geq 0$. This, combined with  $\widetilde v_i=0$ as proven above, leads to the satisfaction of the second line in \eqref{KKT-conditions} for all $i:\,\widetilde x_{i} < U$.

\noindent{\it Case (ii): Consider the case where $\1^\top \zbf> R$.} 

\noindent  By Step 1 of Algorithm \ref{sub-alg}, $(\widetilde x_{(i)})$ is the vector   after sorting the components of $\widetilde \xbf$  in a non-increasing order; that is, $\widetilde x_{(1)}\geq \widetilde x_{(2)}\geq \cdots\geq \widetilde x_{(2d)}$. Also recall that  $(\widetilde v_{(i)})$ is the  vector following the same index order as in $(\widetilde x_{(i)})$. We   let 
\begin{align}\beta =-\frac{\tau}{\gamma}, \text{~and~} 
\mu_{(i)}=\begin{cases}0,& i=1,\cdots, \rho, 
\\
\lambda-\frac{\widetilde x_{(i)}+\tau}{\gamma}, &\text{otherwise},\label{beta mu construct}
\end{cases}
\end{align}
in the KKT conditions \eqref{KKT-conditions}.

We first check the feasibility of $\widetilde \vbf$. We claim that $\widetilde v_{(i)}\geq 0$ for all $i=1,...,2d$. To see this, by the construction in Step 3 of Algorithm \ref{sub-alg}, $\widetilde v_{(i)}=\widetilde x_{(i)}+\tau\geq x_{(\rho)}+\tau = \widetilde v_{(\rho)}\geq U =a\lambda>0$ for $i =1,...,\rho$. Meanwhile, $\widetilde v_{(i)}=0$ for all $i >\rho$.   By the same observation about $\widetilde v_{(i)}$ above, we have $\sum_{i=1}^{2d}\widetilde v_{(i)}=\sum_{i=1}^{\rho} \widetilde v_{(i)}$.  Combining this with the  relationship that $\widetilde v_{(i)}=\widetilde x_{(i)}+\tau$,  we then have
\begin{equation}
    \sum_{i=1}^{2d}\widetilde v_{(i)}=\sum_{i=1}^{\rho} (\widetilde x_{(i)}+\tau)
 =\rho\tau+\sum_{i=1}^{\rho}\widetilde x_{(i)}=R-\sum_{i=1}^{\rho}\widetilde x_{(i)}+\sum_{i=1}^{\rho}\widetilde x_{(i)}=R.\label{ell 1}   
\end{equation}
where the second last equality is due to how $\tau$ is constructed in Step 3 of Algorithm \ref{sub-alg}.
  Therefore, $\widetilde \vbf$ is a feasible solution to \eqref{KKT-conditions}; namely, the last relationship in the fourth line of \eqref{KKT-conditions} holds. 
   By the same reasoning, we immediately have the second relationship in the fourth line \eqref{KKT-conditions} to be satisfied. 

By construction, $\widetilde v_{(i)}=0$ for all $i>\rho$ and $\mu_{(i)}=0$ for all $i\leq \rho$. We then have the third line of \eqref{KKT-conditions} to hold.

For $i=1,\ldots, \rho$, we have $\widetilde v_{(i)} = \widetilde x_{(i)}+\tau\geq U=a\lambda$, $\beta =-\frac{\tau}{\gamma}$, and $\mu_{(i)}=0$, thus
$
 \frac{1}{\gamma}( \widetilde v_{(i)}-\widetilde x_{(i)})+\frac{[a\lambda- \widetilde v_{(i)}]_+}{a}+\beta-\mu_{(i)} = 
\frac{1}{\gamma}(\widetilde v_{(i)}-\widetilde x_{(i)})+\beta=0.
$
For $i =\rho+1,\ldots, 2d$, we have $\widetilde v_{(i)} =0$, $\beta = -\frac{\tau}{\gamma}$, $\mu_{(i)}=\lambda-\frac{\widetilde x_{(i)}+\tau}{\gamma}$, and consequently, 
$
 \frac{1}{\gamma}( \widetilde v_{(i)}-\widetilde x_{(i)})+\frac{[a\lambda- \widetilde v_{(i)}]_+}{a}+\beta-\mu_{(i)} = 
-\frac{1}{\gamma}\widetilde x_{(i)} +\lambda +\beta-\mu_{(i)} = 0,
$ which implies that the first line of \eqref{KKT-conditions} holds.

By the above choices of parameters, it is also easy to verify that $\mu_{(i)}\cdot \widetilde v_{(i)}=0$ for all $i$, which immediately leads to the third line of \eqref{KKT-conditions}. 

To finally verify that $\widetilde \vbf$ is a KKT point, it suffices to show that $\beta\geq 0$ and $\mu_{(i)}\geq 0$ for $i=\rho+1, \cdots, 2d$.  Between them, we first show $\beta\geq 0$. To that end, it suffices to prove $\tau < 0$ by contradiction. For this purpose, we suppose $\tau \geq 0$. Let $k:=\max \{i : \widetilde x_{(i)} \geq U\}$.  Then, by construction, $\widetilde x_{(k)}+\tau\geq\widetilde  x_{(k)}\geq U=a\lambda$ and $\1^\top\zbf=\sum_{i=1}^k\widetilde x_{(i)}$. By definition of $\rho$, we have $\rho\geq k$. Since $\{\widetilde x_{(i)}\}$ is a descent sequence, we have $v_{(i)} = \widetilde x_{(i)}+\tau \geq \widetilde x_{(\rho)} + \tau \geq U=a\lambda >0$ for all  $i:\,i\leq \rho$. Recall that we are considering the case where $\1^\top \widehat \zbf >R$. This contradicts with  \eqref{ell 1}, as $
R = \sum_{i=1}^{2d} v_{(i)}= \sum_{i=1}^{\rho} v_{(i)}
\geq \sum_{i=1}^{k} v_{(i)} = k \tau + \sum_{i=1}^{k} \widetilde x_{(i)}\geq \sum_{i=1}^{k} \widetilde x_{(i)} 
= \mathbf 1^\top \widehat{\mathbf z}.
$
  We have thus proven $\tau < 0$, which evidently leads to $\beta =-\frac{\tau}{\gamma} \geq 0$ in view of \eqref{beta mu construct}.

To show  $\mu_{i}\geq 0$, we recall the construction of $\mu_i$ by \eqref{beta mu construct}. If $\widetilde x_{(\rho+1)}+\tau\leq 0$, then by the positiveness of $\lambda$ and $\gamma$ and the fact that $\{\widetilde x_{(i)}\}$ is a descent sequence, it is easy to see $\mu_{(i)}\geq 0$ for all $i$. Therefore, we only need to consider the case when $\widetilde x_{(\rho+1)}+\tau> 0$ below.  Let $\tau' = \frac{1}{\rho+1}\left(R-\sum_{i=1}^{\rho+1}\widetilde x_{(i)}\right)$, 
 then
$\tau'-\tau =
\frac{1}{\rho+1}\left(R-\sum_{i=1}^{\rho+1}\widetilde x_{(i)}\right) -\frac{1}{\rho}\left(R-\sum_{i=1}^{\rho}\widetilde  x_{(i)} \right)
=\frac{-R-\rho\widetilde  x_{(\rho+1)} + \sum_{i=1}^{\rho}\widetilde x_{(i)} }{\rho(\rho+1)} = -\frac{\widetilde x_{(\rho+1)}+\tau}{\rho+1}$, where the last equality is due to $\tau= \frac{R-\sum_{i=1}^\rho  \widetilde x_{(i)}}{\rho}$ as in Algorithm \ref{sub-alg}.
As a result, $
\widetilde x_{(\rho+1)}  + \tau' - (\tau'-\tau)= \widetilde x_{(\rho+1)} + \tau' +\frac{\widetilde x_{(\rho+1)}+\tau}{\rho+1}
\Longrightarrow
\widetilde x_{(\rho+1)} + \tau= \frac{\rho+1}{\rho}(\widetilde x_{(\rho+1)} +\tau')
$.  Thus, by the definition of  $\rho$, we have 
 $\widetilde x_{(\rho+1)}+\tau' \leq \widetilde x_{(\rho+1)} + \tau < U=a\lambda$.  In view of $\gamma\geq 2a$ and $\frac{\rho+1}{2\rho}\leq  1$,  we further have 
$\frac{\widetilde x_{(\rho+1)}+\tau}{\gamma}\leq  \frac{\widetilde x_{(\rho+1)}+\tau}{2a} =\frac{\rho+1}{2\rho}\cdot\frac{\widetilde x_{(\rho+1)} +\tau'}{a}  < \lambda.$
Since $\{\widetilde x_{(i)}\}$ is a descent sequence, we have $\widetilde x_{(i)}+\tau\leq \widetilde x_{(\rho+1)}+\tau <\gamma \lambda$ for $i=\rho+1, \cdots, d$, which, combined with \eqref{beta mu construct}, proves the non-negativeness of $\mu_{(i)}$. 

In sum, we have proven that $\widetilde\vbf$ is   a KKT point of  \eqref{lemma1-KKT-solution}.

{\bf Step 2.} In this step of the proof, we will show that  the output of Algorithm \ref{sub-alg}, $\vbf$, is the optimal solution to \eqref{lemma1-KKT-solution-3_solution}.   We first observe that $\widetilde{\vbf}$ is   the optimal solution to following convex problem, because its KKT conditions at solution $\widetilde\vbf$ coincide with those  of \eqref{lemma1-KKT-solution}.
\begin{align}\label{lemma1-KKT-solution-2}
  \min_{\wbf=(w_i)\in \R_+^{2d}}\,\left\{\frac{1}{2\gamma}\Vert \wbf-\widetilde{\xbf} \Vert^2+\sum_{i=1}^{2d} \frac{[a\lambda-\widetilde v_{i}]_+}{a}\cdot  w_{i}:\,\mathbf 1^\top\wbf\leq R\right\}:
\end{align}

We claim that if $\widetilde x_{i}=0$ then $\widetilde v_{i}=0$ for any $i$. To see this, suppose the $l$-th entry, $\widetilde x_{l}$, of $\widetilde \xbf$, equals 0. let $\widetilde\vbf'$ be a feasible solution to \eqref{lemma1-KKT-solution-2} and  $\widetilde v_{l}'$ be the $l$-th entry of $\widetilde\vbf$. Suppose that $\widetilde v_{l}'\neq 0$ and it must be that $\widetilde v_{l}'>0$ by its definition. Then $\widetilde \vbf''=\widetilde\vbf'-e_{l}\cdot \widetilde v_l'$ is a strictly better solution than $\widetilde \vbf'$ in terms of the objective value.

We also claim that $\widetilde v_i$ and $\widetilde v_{i+d}$ cannot be nonzero simultaneously since at least one of $\widetilde x_i$ and $\widetilde x_{i+d}$ is zero.  To see this,  recall that  $x_i$ is the $i$-th entry of $\xbf$ for $i=1,...,d$. Then, by definition, $\widetilde x_{i}=\max\{0,\,x_i\}$ and $\widetilde x_{d+i}=\max\{0,\,-x_i\}$, for all $i=1,...,d$. Therefore, it must hold that $\widetilde x_{i}\cdot\widetilde x_{d+i}=0 $ for all $i=1,...,d$. We have shown that  $\widetilde x_{i}=0\Longrightarrow \widetilde v_{i}=0$ for any $i$, thus, at least one of $\widetilde v_i$ and $\widetilde v_{i+d}$ must be zero. By construction, $v_i=\widetilde v_i-\widetilde v_{i+d}$ for all $i$. We thus have
$
\begin{cases}
v_i=\widetilde v_i\geq 0,~\text{and}~\widetilde v_{i+d}=0,&\text{if $x_i\geq 0$},\\
v_i=-\widetilde v_{i+d}\leq 0,~\text{and}~\widetilde v_{i}=0, &\text{if $x_i<0$},
\end{cases}$
which directly gives rise to 
$
\begin{cases}
\vert v_i\vert =\widetilde v_i = \widetilde v_i + \widetilde v_{i+d}, &\text{if $x_i\geq 0$},\\
\vert v_i\vert =\widetilde v_{i+d} = \widetilde v_i + \widetilde v_{i+d}, &\text{if $x_i<0$},
\end{cases}
$ for all $i=1,...,d$.

Thus,  $\widetilde \vbf$ is the optimal solution to  
\begin{align}
\min_{ \wbf=(w_i)}\,&~\sum_{i=1}^d\frac{1}{2\gamma}(w_i-\max\{0,\,x_i\} )^2+\sum_{i=1}^{d}\frac{1}{2\gamma}(w_{i+d} -\max\{0,\,-x_i\} )^2\nonumber
\\\qquad&~~+\sum_{i=1}^{d} \frac{[a\lambda-\vert v_i\vert]_+}{a}\cdot (w_i+w_{i+d})\, \,\nonumber
\\s.t.&~\mathbf 1^\top \wbf \leq R;~~~\wbf\geq \0.\nonumber
\end{align}
Equivalently, $\vbf$ is the optimal solution to 
$\min_{\vbf'\in\R^d}\,\{\sum_{\substack{x_i\geq 0}}\frac{1}{2\gamma}\left(v_i'-x_i \right)^2+\sum_{\substack{x_i< 0}}\frac{1}{2\gamma}\left(v_i'-x_i\right)^2 +\sum_{i=1}^{d} \frac{[a\lambda-\vert v_i\vert]_+}{a}\cdot \vert v_i'\vert:\, \Vert \vbf'\Vert_1 \leq R\},
$ which immediately leads to the desired result. \qed
\end{proof}



\subsection{Proof of Lemma \ref{bounding one particular term}}\label{sec proof lemma here}
 \begin{proof}
Observe that  bounding the left-hand-side of the desired inequality can be reduced to bounding   $\Delta_1$ and $\Delta_2$   below:
{\smaller \begin{align}
 &  \frac{1}{2}\E_{\mathcal V_M}\left[\max_{\widehat S\subset\{1,...,d\}:\,\vert \widehat S\vert\leq \frac{2R}{a\lambda}}\left.\left\Vert \frac{1}{M} \sum_{m=1}^M\frac{f(\xbf+\delta \ubf^{m},\xi^{m})-f(\xbf,\xi^{m})}{\delta}\ubf^{m}_{\widehat S}-\frac{1}{M} \sum_{m=1}^M\nabla_{\widehat S} f(\xbf,\xi^{m})\right\Vert^2\right.\right]\nonumber
  \\  \leq\,& \underbrace{  \E_{\mathcal V_M}\left[\max_{\widehat S\subset\{1,...,d\}:\,\vert \widehat S\vert\leq \frac{2R}{a\lambda}}\left.\left\Vert   \sum_{m=1}^M\left(\frac{f(\xbf+\delta \ubf^{m},\xi^{m})-f(\xbf,\xi^{m})}{M\cdot \delta}\ubf^{m}_{\widehat S}- \frac{\ubf^{m}_{\widehat S}(\ubf^{m})^\top\nabla f(\xbf,\xi^{m})}{M}\right)\right\Vert^2\right.\right]}_{ =: \Delta_1}\nonumber
 \\&  +\underbrace{  \E_{\mathcal V_M}\left[\max_{\widehat S\subset\{1,...,d\}:\,\vert \widehat S\vert\leq \frac{2R}{a\lambda}}\left.\left\Vert \frac{1}{M}\sum_{m=1}^M\nabla_{\widehat S}f(\xbf,\xi^{m})-\frac{1}{M}\sum_{m=1}^M\ubf^{m}_{\widehat S}(\ubf^{m} )^\top\nabla f(\xbf,\xi^{m})\right\Vert^2\right. \right]}_{=:  \Delta_2}\nonumber
 \end{align}}
 
(i) To bound $\Delta_1$,   by Jensen's inequality,
{\smaller
\begin{align}
\Delta_1 \leq \,&\,\E_{\mathcal V_M}\left[\max_{\substack{\widehat S\subset\{1,...,d\}:\,
\\\vert \widehat S\vert\leq \frac{2R}{a\lambda}}}\left.   \frac{1}{M}\sum_{m=1}^M \left\Vert\left(\frac{f(\xbf+\delta \ubf^{m},\xi^{m})-f(\xbf,\xi^{m})}{\delta}-(\ubf^{m})^\top\nabla f(\xbf,\xi^{m})\right)\ubf^{m}_{\widehat S}\right\Vert^2\right.\right]\nonumber
\\  \leq\,&\,\frac{1}{M}\sum_{m=1}^M\E_{\mathcal V_M}\left[\left. \max_{\substack{\widehat S\subset\{1,...,d\}:\,
\\\vert \widehat S\vert\leq \frac{2R}{a\lambda}}}\left\Vert\left(\frac{f(\xbf+\delta \ubf^{m},\xi^{m})-f(\xbf,\xi^{m})-\left\langle\delta\ubf^{m},\,\nabla f(\xbf,\xi^{m})\right\rangle}{\delta}\right)\ubf^{m}_{\widehat S}\right\Vert^2\right.\right]\nonumber
\\  {\leq}\,&\,\frac{L^2\delta^2}{4}\E_{\mathcal V_M}\left[\max_{\substack{\widehat S\subset\{1,...,d\}:\,
\\\vert \widehat S\vert\leq \frac{2R}{a\lambda}}}\left\Vert   \ubf^{m}\right\Vert^4\cdot \vert  \widehat S\vert\right]\leq
\frac{L^2\delta^2}{2}d^2  \cdot \frac{R}{a\lambda}, \nonumber
\end{align}
}\noindent
where the last line  results from three observations: (i) $\Vert \ubf_{\widehat S}^m\Vert^2=\vert \widehat S\vert$; (ii) the Lipschitz continuity of $\nabla f(\,\cdot\,,\xi^{m})$ implies that (as per \eqref{smooth well-known general}) $\vert f(\xbf+\delta\ubf^{m},\xi^{m})-f(\xbf,\xi^{m})-\delta(\ubf^{m})^\top\nabla\nonumber f(\xbf,\xi^{m})\vert\leq\frac{L \delta^2\Vert \ubf^{m}\Vert^2}{2}
$ for almost every $\xi^m$ and every $m=1,...,M$; and (iii) $\Vert\ubf^m\Vert^2=d$.

(ii)  Let $\Xi:=(\xi^{m}:\,1\leq m\leq M)$. To bound $\Delta_2$,  we may invoke Lemma \ref{useful lemma 1} in Section \ref{sec auxiliary} below, together with the independence between $\{\ubf^{m}\}_{1\leq m\leq M}$ and $\nabla f(\xbf,\xi^{m})$, to obtain
 \begin{multline}
   \E_{\Xi} \left[\E_{(\ubf^{m}:\,m=1,...,M)} \left\{\vphantom{v^{v^{v^{v^v}}}_{v_{v_v}}}\underset{1\leq \iota\leq d}{\max} \left[\frac{1}{M}\sum_{m=1}^M\left(\nabla_{\iota}f(\xbf,\xi^{m})-u_{\iota}^{m}(\ubf^{m})^\top \nabla f(\xbf,\xi^{m})\right)\right]^2     \vphantom{v^{v^{v^{v^v}}}_{v_{v_v}}}  \right\}\right]
   \\\leq\E_{\Xi} \left[ \left.\frac{193\sum_{m=1}^M\Vert\nabla f(\xbf,\xi^{m})\Vert^2 }{M^2}\cdot \ln d\right.\right]
 \leq \frac{193(\sigma^2+ \Vert\nabla F(\xbf)\Vert^2 )}{M}\cdot \ln d,\nonumber
 \end{multline}
 where the last inequality is immediately from Assumption \ref{mean and variance}.
 Therefore,
$
 \Delta_2\leq    \,   \frac{386R}{a\lambda}\cdot     \ln d\cdot  \frac{\sigma^2+\Vert\nabla F(\xbf)\Vert^2 }{M}.
$
Combining (i) and (ii) above immediately leads to the desired result. \qed
 \end{proof}
 \subsection{An auxiliary lemma}\label{sec auxiliary}
 \begin{lemma}\label{useful lemma 1}
Let $\{\ubf^m:\,m=1,...,M\}$ be an independent sequence of   $d$-dimensional random vectors whose entries are iid symmetric Bernoulli random variables. Consider a given sequence $\{\boldsymbol\upsilon^m:\,m=1,...,M\}\subset\R^d$.  Let  $u_{i}^m$ and $\upsilon_i^m$ be the $i$-th entries of $\ubf^m$ and $\boldsymbol\upsilon^m$, respectively. Then 
$Z_{\iota}:=\left[M^{-1} \sum_{m=1}^M(\upsilon_{\iota}^m-u_{\iota}^{m}(\ubf^{m})^\top\boldsymbol\upsilon^m)\right]^2$ for any $\iota\in\{1,...,d\}$ is a subexponetial random variable. Furthermore,
$\E\left[\max_{1\leq \iota\leq d}  Z_{\iota} \right]\leq    \frac{193\sum_{m=1}^M\Vert \boldsymbol\upsilon^m\Vert^2 }{M^2}\cdot \ln d.$
 \end{lemma}
 \begin{proof}
 
We will first examine the random variable $\upsilon_{\iota}^m-u_{\iota}^{m}(\ubf^{m})^\top\boldsymbol\upsilon^m$ for a given vector $\boldsymbol\upsilon^m=(\upsilon^m_i)\in\R^d$ and a given index   $\iota\in\{1,...,d\}$.   Observe that 
$u_{\iota}^m\in\{-1,\,1\}\Longrightarrow (u_{\iota}^m)^2=1$, we thus have
that  $\upsilon^m_\iota-u_{\iota}^{m}(\ubf^{m})^\top \boldsymbol\upsilon^m=v_{\iota}^m-(u_{\iota}^{m})^2 \cdot \upsilon^m_{\iota}-\sum_{i\neq \iota}^d u_{\iota}^{m}\cdot u_i^{m}\cdot \upsilon^m_i=\sum_{i\neq \iota} u_{\iota}^{m}\cdot u_i^{m}\cdot \upsilon^m_i$. Thus, $Z_\iota=\left[M^{-1} \sum_{m=1}^M\left(u_{\iota}^{m}\sum_{i\neq \iota} u_{i}^{m}\upsilon_i^m \right)\right]^2$. Below, we prove that  $\left[M^{-1} \sum_{m=1}^M\left(u_{\iota}^{m}\sum_{i\neq \iota} u_{i}^{m}\upsilon_i^m \right)\right]^2$ is a subexponential random variable.

Because $\{u_{i}^{m}\}$ are i.i.d. symmetric Bernoulli random variables, by Hoeffding's inequality  (See Theorem 2.2.2 of \cite{vershynin2018high}) and the fact that $Prob[u_\iota=1]=Prob[u_\iota=-1]
  =0.5$, we have 
\begin{align}
&\begin{rcases*}
   Prob\left\{ \sum_{i\neq \iota} u_{i}^{m} \upsilon^m_i \geq t\right\}\leq  \exp\left(-\frac{t^2}{2\sum_{i\neq \iota} (\upsilon^m_i)^2}\right)\leq  \exp\left(-\frac{t^2}{2\Vert\boldsymbol\upsilon^m\Vert^2}\right)
   \\
   Prob\left\{ \sum_{i\neq \iota} u_{i}^{m} \upsilon^m_i \leq -t\right\}\leq  \exp\left(-\frac{t^2}{2\sum_{i\neq \iota} (\upsilon^m_i)^2}\right)\leq  \exp\left(-\frac{t^2}{2\Vert\boldsymbol\upsilon^m\Vert^2}\right)
   \end{rcases*}{\Longrightarrow}\nonumber
  \\  & 
   Prob\left\{u_{\iota}^{m}\sum_{i\neq\iota} u_{i}^{m} \upsilon^m_i \geq t~\rvert~ u_{\iota}^{m}=1\right\}\cdot 0.5+Prob\left\{ u_{\iota}^{m}\sum_{i\neq\iota} u_{i}^{m} \upsilon^m_i \geq t~\rvert~ u_{\iota}^{m}=-1\right\}\cdot 0.5\nonumber
   \\=&Prob\left\{ u_{\iota}^{m}\sum_{i\neq\iota} u_{i}^{m} \upsilon^m_i \geq t \right\}\leq  \exp\left(-\frac{t^2}{2\Vert\boldsymbol\upsilon^m\Vert^2}\right),\nonumber
   \end{align}
   and likewise, $Prob\{ u_{\iota}^{m}\sum_{i\neq\iota} u_{i}^{m} \upsilon^m_i \leq -t \}\leq  \exp\left(-\frac{t^2}{2\Vert\boldsymbol\upsilon^m\Vert^2}\right)$.
 Therefore, $u_{\iota}^{m}\sum_{i\neq\iota} u_{i}^{m}v_i^m$ is a subgaussian random variable. 
 
 By a well-known property of a subgaussian random variable (as in Lemma 1.5 by \cite{rigollet201518}), it holds that
\begin{align}
  \E\left[\exp\left\{\tau\cdot \left(u_{\iota}^{m}\sum_{i\neq\iota} u_{i}^{m}\upsilon^m_i \right)\right\}\right]\leq \exp\left\{4 \Vert\boldsymbol\upsilon^m\Vert^2 \tau^2\right\},~~\text{for any $\tau\in\R$}.\label{MGF}
\end{align}
In view of the fact that $\{u_{\iota}^{m}\sum_{i\neq\iota} u_{i}^{m}\upsilon^m_i:\,m=1,...,M\}$ is a sequence of independent random variables, we  obtain from  \eqref{MGF} that 
$
  \E\left[\exp\left\{\tau\cdot M^{-1} \sum_{m=1}^M\left(u_{\iota}^{m}\sum_{i\neq \iota} u_{i}^{m}\upsilon^m_i \right)\right\}\right]\leq \exp\left\{\frac{4 \sum_{m=1}^M\Vert\boldsymbol\upsilon^m\Vert^2 \tau^2}{M^2}\right\}.$
   By a well-known relationship between subgaussian and subexponential random variables (as in Lemma 1.12 by \cite{rigollet201518}), we then have that $Z_\iota$ is subexponential in the sense that,  for all $\vert \tau\vert\leq \frac{M^2}{128 \sum_{m=1}^M\Vert\boldsymbol\upsilon^m\Vert^2}$,
\begin{align}  \E[\exp\{\tau Z_\iota-\tau \E[Z_\iota]\}]\leq \exp\left\{128 \tau^2\left( \frac{8\sum_{m=1}^M\Vert\boldsymbol\upsilon^m\Vert^2}{M^{2}}\right)^2\right\} 
\label{mgf of max}\end{align}
which immediately leads to the desired result in the first part of this lemma.

Below we  show the second part of the lemma. For any $\tau:\,0<\tau\leq  \frac{M^2}{128 \sum_{m=1}^M\Vert \boldsymbol\upsilon^{m}\Vert^2}$, we have
$ \E [\max_{1\leq \iota\leq d}  (Z_{\iota}-\E[Z_\iota] )]=  \frac{1}{\tau}\E [\ln \exp \{\tau\cdot \max_{1\leq \iota\leq d}(Z_{\iota}-\E[Z_\iota]) \}  ] 
\leq     \frac{1}{\tau}\ln \E[\exp \{\tau\cdot \max_{1\leq \iota\leq d}(Z_{\iota}-\E[Z_\iota]) \} ] 
 \leq    \frac{1}{\tau}\ln \E[\sum_{1\leq \iota\leq d}\exp\{\tau\cdot (Z_{\iota}-\E[Z_\iota])\}].$

By \eqref{mgf of max},   
$
  \E [\max_{1\leq \iota\leq d}  (Z_{\iota}-\E[Z_\iota])]\leq    \frac{1}{\tau} \ln  (d\cdot \exp \{128 \tau^2 ( 8M^{-2}\sum_{m=1}^M\Vert\boldsymbol\upsilon^m\Vert^2 )^2 \} )\nonumber
=   \frac{\ln d}{\tau}+128 \tau\cdot  ( 8M^{-2}\sum_{m=1}^M\Vert \boldsymbol\upsilon^m\Vert^2 )^2.$
We may as well let $\tau=\frac{M^2}{128 \sum_{m=1}^M\Vert \boldsymbol\upsilon^m\Vert^2 }$. Therefore, 
\begin{align}&\E\left[\max_{1\leq \iota\leq d}  Z_{\iota}\right]-\max_{1\leq \iota\leq d} \E[Z_\iota]   \leq\E\left[\max_{1\leq \iota\leq d} \left(Z_{\iota}-\E[Z_\iota]\right)\right]\nonumber
\\\leq &\frac{128 \sum_{m=1}^M\Vert \boldsymbol\upsilon^m\Vert^2 }{M^2}\cdot \ln d + \frac{64\sum_{m=1}^M\Vert \boldsymbol\upsilon^m \Vert^2}{M^2}.\label{test inequality new 123}
\end{align}

Because $u_{i_1}^{m_1}$ and $u_{i_2}^{m_2}$ are  centered at zero, we know (also by the independence of the random variables) that
\begin{align}
\mathbb E[u_\iota^{m_1} u_\iota^{m_2}  v_{i_1}^{m_1}\upsilon_{i_2}^{m_2}u_{i_1}^{m_1} u_{i_2}^{m_2}]=0,~~~\forall (i_1,\,i_2,\,m_1,\,m_2):\, i_1\neq i_2~\text{or}~m_1\neq m_2.\label{to use zero centered}
\end{align}
Evidently,  for any $\iota:\,1\leq \iota\leq d$, it holds that
\begin{align}
  \E[Z_\iota]&\,=  ~\E\left[M^{-2}\left(\sum_{\substack{m_1,\, m_2}}\,\,\sum_{
\substack{i_1,i_2: \\i_1\neq \iota,\,i_2\neq \iota
}}u_\iota^{m_1} u_\iota^{m_2}  v_{i_1}^{m_1}\upsilon_{i_2}^{m_2}u_{i_1}^{m_1} u_{i_2}^{m_2}\right)\right]\nonumber
\\&\stackrel{\eqref{to use zero centered}}{=}   ~ \E\left[M^{-2}\sum_{\substack{i\neq \iota\\ 1\leq m\leq M}}(u_\iota^{m})^2 (u_i^{m}v_i^m)^2\right]
\leq    M^{-2}\sum_{m=1}^M\Vert \boldsymbol\upsilon^m\Vert^2\label{expectation upper bound}
\end{align}
 The last inequality above is due to $u_{i}^{m}\in\{-1,1\}$ for all  $i$ and $m$.
 
Combining \eqref{test inequality new 123}, \eqref{expectation upper bound}, and the assumption that $d\geq 3$, we have 
$ 
  \E [\max_{1\leq \iota\leq d}  Z_{\iota}  ]\leq \frac{128 \sum_{m=1}^M\Vert \boldsymbol\upsilon^m\Vert^2 }{M^2} \ln d + \frac{65\sum_{m=1}^M\Vert \boldsymbol\upsilon^m\Vert^2}{M^2}\leq \frac{193 \sum_{m=1}^M\Vert \boldsymbol\upsilon^m\Vert^2 }{M^2} \ln d$, which completes the proof. \qed
 \end{proof}
 
\end{appendix}

 \bibliographystyle{abbrv}
 \bibliography{references.bib} 

\newpage

\begin{appendix}
\section*{\centering Supplemental Material}

 \begin{table}[H] 
\caption{Comparison of suboptimality gap when the mini-batch size $M$ increases.  Here, the budget of the total number of zeroth-order oracle calls is 160,000, the dimension $d=2^{15}=32,768$,  the iteration count $K=\left\lfloor\frac{320,000}{M} \right\rfloor$, and $\delta = 1\times 10^{-7}$. For each value of $M$, numbers in the first row  are the average suboptimality gaps out of ten random replications while those in  second (behind the ``{\scriptsize$\pm$}''-signs) are the standard deviations. ``{\it e}\,$	\bullet$'' means ``$\times 10^\bullet$''.}\label{tab: increase M}
\begin{center}
\begin{tabular}{|c|rrr|rrr|r|}
\toprule
{$M$} 
                  &si-sgf$^{R}$          & si-sgf$^{*}$       & si-sgf$^{A}$             &si-sgfs$^{R}$          & si-sgfs$^{*}$          & si-sgfs$^{A}$        & sgf \\\midrule
\multirow{2}{*}{40} & 1.3e-1 & 1.1e-1 & 1.7e-1 & 5.5e-1 & 3.0e-1 & 3.8e-1 & 1.5 \\ 
 & {\scriptsize$\pm$}3.0e-2 & {\scriptsize$\pm$}4.7e-3 & {\scriptsize$\pm$}1.7e-3 & {\scriptsize$\pm$}5.3e-1 & {\scriptsize$\pm$}3.0e-2 & {\scriptsize$\pm$}5.6e-3 & {\scriptsize$\pm$}7.6e-3 \\\hline
\multirow{2}{*}{100} & 4.5e-1 & 3.3e-2 & 1.2e-1 & 9.8e-2 & 8.8e-2 & 9.6e-2 & 1.4 \\ 
 & {\scriptsize$\pm$}1.2 & {\scriptsize$\pm$}3.1e-3 & {\scriptsize$\pm$}2.0e-3 & {\scriptsize$\pm$}7.6e-3 & {\scriptsize$\pm$}6.4e-3 & {\scriptsize$\pm$}1.6e-3 & {\scriptsize$\pm$}4.8e-3 \\\hline
\multirow{2}{*}{160} & 1.0e-1 & 3.1e-2 & 1.5e-1 & 9.0e-2 & 5.1e-2 & 5.6e-2 & 1.4 \\ 
 & {\scriptsize$\pm$}1.5e-1 & {\scriptsize$\pm$}3.3e-3 & {\scriptsize$\pm$}2.1e-3 & {\scriptsize$\pm$}1.1e-1 & {\scriptsize$\pm$}4.9e-3 & {\scriptsize$\pm$}1.1e-3 & {\scriptsize$\pm$}6.6e-3 \\\hline
\multirow{2}{*}{220} & 4.0e-1 & 3.3e-2 & 2.0e-1 & 6.7e-2 & 4.2e-2 & 4.2e-2 & 1.4 \\ 
 & {\scriptsize$\pm$}3.9e-1 & {\scriptsize$\pm$}1.9e-3 & {\scriptsize$\pm$}3.0e-3 & {\scriptsize$\pm$}6.7e-2 & {\scriptsize$\pm$}2.6e-3 & {\scriptsize$\pm$}1.4e-3 & {\scriptsize$\pm$}7.9e-3 \\\hline
\multirow{2}{*}{280} & 7.6e-1 & 3.8e-2 & 2.4e-1 & 4.4e-2 & 3.5e-2 & 3.6e-2 & 1.5 \\ 
 & {\scriptsize$\pm$}1.6 & {\scriptsize$\pm$}1.3e-3 & {\scriptsize$\pm$}4.0e-3 & {\scriptsize$\pm$}1.1e-2 & {\scriptsize$\pm$}1.5e-3 & {\scriptsize$\pm$}4.7e-4 & {\scriptsize$\pm$}6.4e-3 \\\hline
\multirow{2}{*}{340} & 1.5 & 4.4e-2 & 2.9e-1 & 4.8e-2 & 3.7e-2 & 3.6e-2 & 1.5 \\ 
 & {\scriptsize$\pm$}2.3 & {\scriptsize$\pm$}1.5e-3 & {\scriptsize$\pm$}2.5e-3 & {\scriptsize$\pm$}1.1e-2 & {\scriptsize$\pm$}1.5e-3 & {\scriptsize$\pm$}8.5e-4 & {\scriptsize$\pm$}8.7e-3 \\\hline
\multirow{2}{*}{400} & 6.2e-1 & 5.3e-2 & 3.5e-1 & 5.8e-2 & 4.5e-2 & 4.0e-2 & 1.5 \\ 
 & {\scriptsize$\pm$}1.4 & {\scriptsize$\pm$}1.8e-3 & {\scriptsize$\pm$}2.3e-3 & {\scriptsize$\pm$}2.3e-2 & {\scriptsize$\pm$}1.4e-3 & {\scriptsize$\pm$}1.1e-3 & {\scriptsize$\pm$}6.9e-3 \\\hline
\multirow{2}{*}{460} & 7.1e-1 & 6.6e-2 & 4.1e-1 & 3.9e-1 & 5.4e-2 & 4.5e-2 & 1.5 \\ 
 & {\scriptsize$\pm$}9.7e-1 & {\scriptsize$\pm$}1.3e-3 & {\scriptsize$\pm$}4.1e-3 & {\scriptsize$\pm$}1.0 & {\scriptsize$\pm$}2.1e-3 & {\scriptsize$\pm$}1.3e-3 & {\scriptsize$\pm$}7.4e-3 \\\hline
\multirow{2}{*}{520} & 6.6e-1 & 6.8e-2 & 4.9e-1 & 1.3e-1 & 2.5e-2 & 5.0e-2 & 1.5 \\ 
 & {\scriptsize$\pm$}1.2 & {\scriptsize$\pm$}2.0e-3 & {\scriptsize$\pm$}3.6e-3 & {\scriptsize$\pm$}2.0e-1 & {\scriptsize$\pm$}1.3e-3 & {\scriptsize$\pm$}1.3e-3 & {\scriptsize$\pm$}9.1e-3 \\\hline
\multirow{2}{*}{580} & 2.9e-1 & 7.5e-2 & 5.6e-1 & 3.9e-1 & 2.1e-2 & 5.7e-2 & 1.5 \\ 
 & {\scriptsize$\pm$}3.6e-1 & {\scriptsize$\pm$}2.2e-3 & {\scriptsize$\pm$}3.4e-3 & {\scriptsize$\pm$}1.1 & {\scriptsize$\pm$}2.6e-3 & {\scriptsize$\pm$}2.6e-3 & {\scriptsize$\pm$}5.7e-3 \\\hline
\multirow{2}{*}{640} & 7.1e-1 & 8.9e-2 & 6.4e-1 & 5.7e-2 & 2.5e-2 & 6.6e-2 & 1.5 \\ 
 & {\scriptsize$\pm$}1.0 & {\scriptsize$\pm$}1.7e-3 & {\scriptsize$\pm$}5.3e-3 & {\scriptsize$\pm$}5.8e-2 & {\scriptsize$\pm$}1.9e-3 & {\scriptsize$\pm$}2.7e-3 & {\scriptsize$\pm$}7.1e-3 \\\hline
\multirow{2}{*}{700} & 1.0 & 1.1e-1 & 7.3e-1 & 4.7e-1 & 2.8e-2 & 7.4e-2 & 1.5 \\ 
 & {\scriptsize$\pm$}8.6e-1 & {\scriptsize$\pm$}3.4e-3 & {\scriptsize$\pm$}6.1e-3 & {\scriptsize$\pm$}1.3 & {\scriptsize$\pm$}8.3e-4 & {\scriptsize$\pm$}1.1e-3 & {\scriptsize$\pm$}4.4e-3 \\\hline
\multirow{2}{*}{760} & 1.2 & 1.2e-1 & 8.1e-1 & 6.0e-1 & 3.0e-2 & 8.2e-2 & 1.5 \\ 
 & {\scriptsize$\pm$}1.2 & {\scriptsize$\pm$}1.2e-3 & {\scriptsize$\pm$}4.3e-3 & {\scriptsize$\pm$}1.1 & {\scriptsize$\pm$}1.3e-3 & {\scriptsize$\pm$}1.4e-3 & {\scriptsize$\pm$}8.1e-3 \\\hline
\multirow{2}{*}{820} & 9.3e-1 & 1.4e-1 & 8.9e-1 & 1.6e-1 & 3.2e-2 & 8.8e-2 & 1.5 \\ 
 & {\scriptsize$\pm$}1.6 & {\scriptsize$\pm$}1.5e-3 & {\scriptsize$\pm$}5.6e-3 & {\scriptsize$\pm$}1.9e-1 & {\scriptsize$\pm$}1.3e-3 & {\scriptsize$\pm$}1.5e-3 & {\scriptsize$\pm$}6.6e-3 \\\hline
\multirow{2}{*}{880} & 1.4 & 1.6e-1 & 9.9e-1 & 2.2e-1 & 3.4e-2 & 9.7e-2 & 1.5 \\ 
 & {\scriptsize$\pm$}1.8 & {\scriptsize$\pm$}4.8e-3 & {\scriptsize$\pm$}1.1e-2 & {\scriptsize$\pm$}4.8e-1 & {\scriptsize$\pm$}1.5e-3 & {\scriptsize$\pm$}3.4e-3 & {\scriptsize$\pm$}7.2e-3   
\\\bottomrule
\end{tabular}
\end{center}
\end{table}

\begin{table} 
\caption{Comparison of suboptimality gap when $d$ increases exponentially.  $\delta=10^{-7}$. The budget of the total number of zeroth-order oracle calls is  160,000. For all variants of {\it si-sgf} and {\it si-sgfs}, the mini-batch sizes $M$ are chosen as 160 and 280, respectively. Correspondingly,  iteration count $K=\left\lfloor\frac{320,000}{M} \right\rfloor$.  For each value of $d$, numbers in the first row  are the average suboptimality gaps out of five random replications while those in  second (behind the ``{\scriptsize$\pm$}''-signs) are the standard deviations.  ``{\it e}\,$	\bullet$'' means ``$\times 10^\bullet$''. The numbers in bold refer to the smallest average suboptimality gaps for the same $d$.}\label{tab: increase d}
\begin{center}
\begin{tabular}{|c|rrr|rrr|r|}
\toprule
{$d$} 
                                   &si-sgf$^{R}$          & si-sgf$^{*}$       & si-sgf$^{A}$             &si-sgfs$^{R}$          & si-sgfs$^{*}$          & si-sgfs$^{A}$        & sgf \\\midrule
\multirow{2}{*}{2$^{6}$} & 4.2e-2 & 3.2e-2 & 8.1e-2 & 1.5e-2 & 1.5e-2 & 2.5e-2 & {\bf 2.2e-3} \\ 
 & {\scriptsize$\pm$}1.4e-2 & {\scriptsize$\pm$}7.0e-3 & {\scriptsize$\pm$}2.2e-3 & {\scriptsize$\pm$}3.1e-3 & {\scriptsize$\pm$}3.6e-3 & {\scriptsize$\pm$}1.3e-3 & {\scriptsize$\pm$}7.4e-4 \\\hline
\multirow{2}{*}{2$^{7}$} & 3.5e-2 & 3.0e-2 & 8.1e-2 & 3.2e-2 & 1.9e-2 & 2.6e-2 & {\bf 3.9e-3} \\ 
 & {\scriptsize$\pm$}7.4e-3 & {\scriptsize$\pm$}4.7e-3 & {\scriptsize$\pm$}8.0e-4 & {\scriptsize$\pm$}1.5e-2 & {\scriptsize$\pm$}3.2e-3 & {\scriptsize$\pm$}1.7e-3 & {\scriptsize$\pm$}6.1e-4 \\\hline
\multirow{2}{*}{2$^{8}$} & 3.9e-2 & 3.7e-2 & 8.2e-2 & 4.0e-2 & 2.7e-2 & 2.7e-2 & {\bf 1.0e-2} \\ 
 & {\scriptsize$\pm$}1.1e-2 & {\scriptsize$\pm$}4.9e-3 & {\scriptsize$\pm$}1.4e-3 & {\scriptsize$\pm$}1.3e-2 & {\scriptsize$\pm$}3.1e-3 & {\scriptsize$\pm$}1.9e-3 & {\scriptsize$\pm$}8.1e-4 \\\hline
\multirow{2}{*}{2$^{9}$} & 7.6e-2 & 3.7e-2 & 8.4e-2 & 4.9e-2 & 3.7e-2 & 2.6e-2 & {\bf 2.2e-2} \\ 
 & {\scriptsize$\pm$}3.8e-2 & {\scriptsize$\pm$}2.1e-3 & {\scriptsize$\pm$}1.8e-3 & {\scriptsize$\pm$}1.0e-2 & {\scriptsize$\pm$}2.2e-3 & {\scriptsize$\pm$}1.2e-3 & {\scriptsize$\pm$}6.5e-4 \\\hline
\multirow{2}{*}{2$^{10}$} & 5.0e-1 & 4.1e-2 & 8.5e-2 & 5.6e-2 & 4.5e-2 & {\bf 2.5e-2} & 4.4e-2 \\ 
 & {\scriptsize$\pm$}9.9e-1 & {\scriptsize$\pm$}2.3e-3 & {\scriptsize$\pm$}1.2e-3 & {\scriptsize$\pm$}6.3e-3 & {\scriptsize$\pm$}1.7e-3 & {\scriptsize$\pm$}1.5e-3 & {\scriptsize$\pm$}4.9e-4 \\\hline
\multirow{2}{*}{2$^{11}$} & 6.9e-2 & 3.7e-2 & 8.8e-2 & 2.1e-1 & 4.4e-2 & {\bf 2.4e-2} & 7.2e-2 \\ 
 & {\scriptsize$\pm$}4.1e-2 & {\scriptsize$\pm$}2.8e-3 & {\scriptsize$\pm$}2.1e-3 & {\scriptsize$\pm$}3.4e-1 & {\scriptsize$\pm$}2.2e-3 & {\scriptsize$\pm$}6.7e-4 & {\scriptsize$\pm$}1.0e-3 \\\hline
\multirow{2}{*}{2$^{12}$} & 2.1e-1 & 3.4e-2 & 9.8e-2 & 1.1e-1 & 4.1e-2 & {\bf 2.5e-2} & 8.7e-2 \\ 
 & {\scriptsize$\pm$}3.0e-1 & {\scriptsize$\pm$}2.9e-3 & {\scriptsize$\pm$}1.5e-3 & {\scriptsize$\pm$}1.4e-1 & {\scriptsize$\pm$}2.1e-3 & {\scriptsize$\pm$}7.6e-4 & {\scriptsize$\pm$}1.9e-3 \\\hline
\multirow{2}{*}{2$^{13}$} & 1.8e-1 & 3.3e-2 & 1.1e-1 & 4.2e-2 & 3.6e-2 & {\bf 2.5e-2} & 1.4e-1 \\ 
 & {\scriptsize$\pm$}2.8e-1 & {\scriptsize$\pm$}4.7e-3 & {\scriptsize$\pm$}1.1e-3 & {\scriptsize$\pm$}7.4e-3 & {\scriptsize$\pm$}1.4e-3 & {\scriptsize$\pm$}1.3e-3 & {\scriptsize$\pm$}1.2e-3 \\\hline
\multirow{2}{*}{2$^{14}$} & 2.6e-1 & {\bf 2.7e-2} & 1.3e-1 & 5.5e-1 & 3.4e-2 & 3.0e-2 & 4.5e-1 \\ 
 & {\scriptsize$\pm$}3.6e-1 & {\scriptsize$\pm$}2.2e-3 & {\scriptsize$\pm$}1.0e-3 & {\scriptsize$\pm$}1.1 & {\scriptsize$\pm$}1.1e-3 & {\scriptsize$\pm$}6.1e-4 & {\scriptsize$\pm$}5.0e-3 \\\hline
\multirow{2}{*}{2$^{15}$} & 5.3e-2 & {\bf 3.0e-2} & 1.6e-1 & 3.9e-2 & 3.4e-2 & 3.5e-2 & 1.5 \\ 
 & {\scriptsize$\pm$}2.4e-2 & {\scriptsize$\pm$}2.8e-3 & {\scriptsize$\pm$}3.5e-3 & {\scriptsize$\pm$}1.0e-2 & {\scriptsize$\pm$}1.3e-3 & {\scriptsize$\pm$}8.2e-4 & {\scriptsize$\pm$}4.8e-3 \\\hline
\multirow{2}{*}{2$^{16}$} & 8.0e-2 & {\bf 3.5e-2} & 1.9e-1 & 1.1 & 3.6e-2 & 4.3e-2 & 3.0 \\ 
 & {\scriptsize$\pm$}7.0e-2 & {\scriptsize$\pm$}4.1e-3 & {\scriptsize$\pm$}1.8e-3 & {\scriptsize$\pm$}2.3 & {\scriptsize$\pm$}1.9e-3 & {\scriptsize$\pm$}8.6e-4 & {\scriptsize$\pm$}3.4e-3 \\\hline
\multirow{2}{*}{2$^{17}$} & 7.0e-1 & {\bf 3.8e-2} & 2.2e-1 & 5.9e-2 & 4.1e-2 & 5.1e-2 & 4.4 \\ 
 & {\scriptsize$\pm$}1.1 & {\scriptsize$\pm$}3.7e-3 & {\scriptsize$\pm$}1.0e-3 & {\scriptsize$\pm$}2.5e-2 & {\scriptsize$\pm$}1.6e-3 & {\scriptsize$\pm$}1.4e-3 & {\scriptsize$\pm$}3.4e-3 \\\hline
\multirow{2}{*}{2$^{18}$} & 3.7e-1 & {\bf 4.3e-2} & 2.6e-1 & 7.3e-2 & {\bf 4.3e-2} & 6.1e-2 & 5.5 \\ 
 & {\scriptsize$\pm$}4.6e-1 & {\scriptsize$\pm$}4.2e-3 & {\scriptsize$\pm$}2.5e-3 & {\scriptsize$\pm$}3.6e-2 & {\scriptsize$\pm$}3.4e-3 & {\scriptsize$\pm$}9.9e-4 & {\scriptsize$\pm$}2.8e-3 \\\hline
\multirow{2}{*}{2$^{19}$} & 5.9e-1 & 5.6e-2 & 3.0e-1 & 1.3 & {\bf 4.7e-2} & 7.2e-2 & 6.1 \\ 
 & {\scriptsize$\pm$}9.7e-1 & {\scriptsize$\pm$}2.6e-3 & {\scriptsize$\pm$}3.6e-3 & {\scriptsize$\pm$}2.8 & {\scriptsize$\pm$}4.3e-3 & {\scriptsize$\pm$}1.4e-3 & {\scriptsize$\pm$}9.2e-4 \\\hline
\multirow{2}{*}{2$^{20}$} & 1.6e-1 & 7.0e-2 & 3.5e-1 & 3.5e-1 & {\bf 5.2e-2} & 8.6e-2 & 6.4 \\ 
 & {\scriptsize$\pm$}7.4e-2 & {\scriptsize$\pm$}1.2e-3 & {\scriptsize$\pm$}3.3e-3 & {\scriptsize$\pm$}4.1e-1 & {\scriptsize$\pm$}1.7e-3 & {\scriptsize$\pm$}1.7e-3 & {\scriptsize$\pm$}8.3e-4 \\\hline
\multirow{2}{*}{2$^{21}$} & 4.5e-1 & 9.0e-2 & 4.0e-1 & 2.5e-1 & {\bf 5.8e-2} & 9.8e-2 & 6.6 \\ 
 & {\scriptsize$\pm$}7.5e-1 & {\scriptsize$\pm$}3.3e-3 & {\scriptsize$\pm$}5.2e-3 & {\scriptsize$\pm$}4.0e-1 & {\scriptsize$\pm$}2.6e-3 & {\scriptsize$\pm$}2.1e-3 & {\scriptsize$\pm$}3.3e-4  
  \\\bottomrule
\end{tabular}
\end{center}
\end{table}

  \begin{table} 
\caption{Comparison of suboptimality gap when $\delta$ increases exponentially.   Here, the budget of the total number of zeroth-order oracle calls is 160,000, and the dimension $d=2^{15}=32,768$. For all variants of {\it si-sgf} and {\it si-sgfs}, the mini-batch sizes $M$ are chosen as 160 and 280, respectively. Correspondingly,  iteration count $K=\left\lfloor\frac{320,000}{M} \right\rfloor$.  For each value of $\delta$, numbers in the first row  are the average suboptimality gaps out of ten random replications while those in  second (behind the ``{\scriptsize$\pm$}''-signs) are the standard deviations.  ``{\it e}\,$	\bullet$'' means ``$\times 10^\bullet$''.}\label{tab: increase delta}
\begin{center}
\begin{tabular}{|c|rrr|rrr|r|}
\toprule
{$\delta$} 
                                   &si-sgf$^{R}$          & si-sgf$^{*}$       & si-sgf$^{A}$             &si-sgfs$^{R}$          & si-sgfs$^{*}$          & si-sgfs$^{A}$        & sgf \\\midrule
 
\multirow{2}{*}{\small10$^{-8}$} & 3.1e-1 & 2.9e-2 & 1.6e-1 & 3.8e-2 & 3.6e-2 & 3.6e-2 & 1.6 \\ 
 & {\scriptsize$\pm$}5.8e-1 & {\scriptsize$\pm$}1.9e-3 & {\scriptsize$\pm$}2.3e-3 & {\scriptsize$\pm$}4.1e-3 & {\scriptsize$\pm$}1.5e-3 & {\scriptsize$\pm$}7.2e-4 & {\scriptsize$\pm$}1.9e-1 \\\hline
\multirow{2}{*}{\small10$^{-7}$} & 3.9e-1 & 2.8e-2 & 1.6e-1 & 4.5e-2 & 3.6e-2 & 3.6e-2 & 1.6 \\ 
 & {\scriptsize$\pm$}7.4e-1 & {\scriptsize$\pm$}2.5e-3 & {\scriptsize$\pm$}2.0e-3 & {\scriptsize$\pm$}1.3e-2 & {\scriptsize$\pm$}2.4e-3 & {\scriptsize$\pm$}9.7e-4 & {\scriptsize$\pm$}8.2e-2 \\\hline
\multirow{2}{*}{\small10$^{-6}$} & 9.9e-2 & 3.3e-2 & 1.6e-1 & 3.9e-2 & 3.5e-2 & 3.5e-2 & 1.6 \\ 
 & {\scriptsize$\pm$}9.5e-2 & {\scriptsize$\pm$}4.7e-3 & {\scriptsize$\pm$}1.1e-3 & {\scriptsize$\pm$}4.6e-3 & {\scriptsize$\pm$}3.2e-3 & {\scriptsize$\pm$}9.1e-4 & {\scriptsize$\pm$}1.6e-1 \\\hline
\multirow{2}{*}{\small10$^{-5}$} & 1.1e-1 & 3.0e-2 & 1.6e-1 & 8.6e-2 & 3.6e-2 & 3.7e-2 & 1.5 \\ 
 & {\scriptsize$\pm$}1.1e-1 & {\scriptsize$\pm$}4.0e-3 & {\scriptsize$\pm$}2.4e-3 & {\scriptsize$\pm$}7.3e-2 & {\scriptsize$\pm$}5.8e-4 & {\scriptsize$\pm$}1.1e-3 & {\scriptsize$\pm$}4.9e-2 \\\hline
\multirow{2}{*}{\small10$^{-4}$} & 6.2e-1 & 1.1e-1 & 3.2e-1 & 1.4e-1 & 1.1e-1 & 1.5e-1 & 1.6 \\ 
 & {\scriptsize$\pm$}6.3e-1 & {\scriptsize$\pm$}8.1e-3 & {\scriptsize$\pm$}6.5e-3 & {\scriptsize$\pm$}1.7e-2 & {\scriptsize$\pm$}4.6e-3 & {\scriptsize$\pm$}2.3e-3 & {\scriptsize$\pm$}6.4e-2 \\\hline
\multirow{2}{*}{\small10$^{-3}$} & 6.8 & 6.6 & 6.7 & 6.9 & 6.4 & 6.7 & 6.7 \\ 
 & {\scriptsize$\pm$}2.8e-2 & {\scriptsize$\pm$}5.2e-2 & {\scriptsize$\pm$}1.6e-3 & {\scriptsize$\pm$}1.5e-1 & {\scriptsize$\pm$}1.5e-1 & {\scriptsize$\pm$}6.1e-3 & {\scriptsize$\pm$}2.6e-4 
  \\\bottomrule
\end{tabular}
\end{center}
\end{table}
\end{appendix}
\end{document}